\documentclass[pdflatex,sn-mathphys-num]{sn-jnl}
\usepackage{amsmath}
\usepackage{amsthm}
\usepackage{amsfonts}
\usepackage{amssymb}
\usepackage{hyperref}
\usepackage{tabularx}
\usepackage{bbm}
\usepackage{adjustbox}
\usepackage{floatpag}
\numberwithin{equation}{section}
\usepackage[mathscr]{euscript}
\usepackage{pdflscape}
\usepackage{graphicx,tabularx,tabulary,booktabs}
\usepackage{algorithm}
\usepackage{algpseudocode} 
\usepackage{enumitem} 
\usepackage{pst-node,tikz-cd}
\allowdisplaybreaks
\usepackage{geometry}
\usepackage{makecell}
\usepackage{pifont} 
\usepackage{array}        
\usepackage{booktabs}     
\usepackage{float}        

\usepackage{afterpage}
\usepackage{mhchem}
\usepackage{color}
\usepackage{tcolorbox}
\usepackage{graphicx}
\usepackage[title]{appendix}
\usepackage{comment}
\usepackage{caption}
\usepackage{subcaption}
\theoremstyle{definition} 
\newtheorem{definition}{Definition}[section]
\newtheorem{observation}{Observation}[section]
\newtheorem{example}{Example}[section]
\newtheorem{assumption}{Assumption}[section]

\theoremstyle{plain} 
\newtheorem{theorem}{Theorem}[section]
\newtheorem{lemma}{Lemma}[section]
\newtheorem{proposition}{Proposition}[section]
\newtheorem{corollary}{Corollary}[section]

\newcommand{\conv}{\mathrm{conv}}

\newcommand{\proj}{\mathrm{Proj}}
\newcommand{\opt}{\texttt{opt}}

\newcommand{\mR}{\mathbb{R}}
\newcommand{\X}{X}

\newcommand{\cC}{\mathcal{C}}

\newcommand{\cT}{\mathcal{T}}

\newcommand{\cZ}{\mathcal{Z}}
\newcommand{\bs}[1]{\boldsymbol{#1}}

\newcommand{\E}{\mathbb{E}}

\newcommand{\norm}[1]{\left\lVert#1\right\rVert}

\newcommand{\inner}[2]{\left\langle #1,#2 \right\rangle}

\newcommand{\mudec}{\mu_{\mathrm{dec}}}

\newcommand{\Thetadisc}{\Theta_{\mathrm{disc}}}
\newcommand{\alphamax}{\alpha_{\mathrm{max}}}

\renewcommand{\sp}{\mathrm{span}}

\newcommand{\gap}{\mathrm{Gap}}

\newcommand{\xtr}{x_\mathrm{trial}}

\newcommand{\mB}{\mathbb{B}}
\newcommand{\cB}{\mathcal{B}}

\newcommand{\cM}{\mathcal{M}}
\newcommand{\cD}{\mathcal{D}}
\newcommand{\tD}{\widetilde{\mathcal{D}}}
\newcommand{\tM}{\widetilde{\mathcal{M}}}

\newcommand{\QPt}[2]{\mathrm{QP}\textrm{-}{#1}\textrm{-}{#2}}
\newcommand{\sol}[1]{\mathrm{sol(}#1\mathrm{)}}
\newcommand{\vc}[1]{\mathrm{vec(}#1\mathrm{)}}

\newcommand{\qpvc}[2]{\vc{\QPt{#1}{#2}}}

\newcommand{\epsilonopt}{\epsilon_\textrm{opt}}
\newcommand{\nuopt}{\nu_\textrm{opt}}

\newcommand{\bdm}{\begin{displaymath}}
\newcommand{\edm}{\end{displaymath}}
\newcommand{\beq}{\begin{equation}}
\newcommand{\eeq}{\end{equation}}
\newcommand{\ba}{\begin{aligned}}
\newcommand{\ea}{\end{aligned}}
\newcommand{\expref}[1]{Example~\ref{#1}}
\newcommand{\figref}[1]{Figure~\ref{#1}}
\newcommand{\assref}[1]{Assumption~\ref{#1}}
\newcommand{\thmref}[1]{Theorem~\ref{#1}}
\newcommand{\lemref}[1]{Lemma~\ref{#1}}
\newcommand{\propref}[1]{Proposition~\ref{#1}}

\newcommand{\defref}[1]{Definition~\ref{#1}}
\newcommand{\secref}[1]{Section~\ref{#1}}
\newcommand{\obsref}[1]{Observation~\ref{#1}}
\newcommand{\tabref}[1]{Table~\ref{#1}}
\newcommand{\algoref}[1]{Algorithm~\ref{#1}}
\newcommand{\linref}[1]{Line~\ref{#1}}

\begin{document}
\title[Article Title]{Structured Nonsmooth Optimization Using Functional Encoding and Branching Information}
\author{Fengqiao Luo}\email{fengqiaoluo2014@u.northwestern.edu}

\affil{\orgdiv{Department of Industrial Engineering and Management Science}, \orgname{Northwestern University}, \orgaddress{\state{IL}, \country{United States}}}

\abstract{
We develop a novel gradient-based algorithm for optimizing nonsmooth nonconvex functions where nonsmoothness arises from explicit nonsmooth operators in the objective's analytical form. Our key innovation involves encoding active smooth branches of these operators, enabling both branch function extraction at arbitrary points and transition detection through branch tracking. This approach yields a Branch-Information-Driven Gradient Descent (BIGD) method for encodable piecewise-differentiable functions, with an enhanced version achieving local linear convergence under appropriate conditions. The computationally efficient encoding mechanism is straightforward to implement. The power of using branch information has been proved via substantial numerical experiments  
 compared to some existing nonsmooth optimization methods on standard test problems. Most importantly, for piecewise-smooth problems given analytical expressions,
 implementation of functional encoding can be integrated into a wide range of existing nonsmooth optimization methods to improve the bundle points management, 
 reduce the complexity of the quadratic programming sub-problems, and improve the efficiency of line search.
}


\maketitle

\section{Introduction}
We study the unconstrained optimization problem
\begin{equation}\label{opt:piecewise-smooth}
    \min_{x \in \mathbb{R}^n} f(x),
\end{equation}
where $f$ is continuous but potentially nonconvex and nonsmooth. While general nonsmooth functions may be non-differentiable on dense subsets, we focus on functions whose nonsmoothness stems from explicit nonsmooth operators, including $\max(\cdot)$, $\min(\cdot)$, $|\cdot|$, $(\cdot)_+$, and various rule-based operators that appear in $f$'s analytical form. 
This important class of piecewise-differentiable functions covers numerous practical applications such as nonsmooth penalized objectives in signal processing and machine learning, piecewise nonlinear regression problems, control systems, and deep neural networks. 
Throughout this work, we assume algorithms have direct access to $f$'s analytical form, enabling exploitation of its underlying piecewise structure.

\subsection{Motivation}
Existing methods for nonsmooth optimization include subgradient methods, bundle methods, trust region methods, gradient sampling methods, and quasi-Newton methods. These approaches typically treat the objective as a general non-differentiable function (e.g., locally Lipschitz continuous) where subgradient information is available almost everywhere. While this general framework makes the methods broadly applicable, it fails to fully exploit the analytical structure present in many practical problems of interest.

For the problem class in \eqref{opt:piecewise-smooth}, each nonsmooth operator may induce multiple branch functions, potentially containing further nested nonsmooth operators. Crucially, we can numerically track the active branches of each nonsmooth operator to determine the locally smooth function at any point $x$, with the branching sequence being encodable. This suggests an algorithmic opportunity: by recording the active branch codes at previous iterates, could we develop more efficient optimization methods?

This paper systematically addresses this question through theoretical analysis, algorithmic development, and practical implementation. We demonstrate how branch-coding information can significantly enhance optimization efficiency for piecewise-smooth functions.

\subsection{Review of related works}
Bundle methods are in the category of the most efficient methods for nonsmooth optimization,
which were initiated in \cite{lemarechal1977-bundle-methods}, and have been developed to incorporate more ingredients
to increase stability and efficiency since then. The key idea of bundle methods is to utilize the subgradient
information obtained from previous iterations to build the current local model and determine the next searching direction.
A typical bundle method often consists of five major steps in an iteration: construction of the local model, 
solving the local model to determine the searching direction, line search for getting an efficient while reliable step size,
checking the termination criteria and updating the set of bundle points for the next iteration \cite{bagirov2014-bundle-method}.
The quality and reliability of the local model, as well as the complexity of constructing it 
play an essential role in determining the algorithm performance. To improve the quality 
of the local model, primary bundle methods have been integrated with other methods to develop advanced 
methods such as proximal bundle 
\cite{hintermuller2001-prox-bundle-approx-subgrad,kiwiel2006-prox-bundle-approx-subgrad-linearization,
hare2010-redistributed-prox-bundle,sagastizabal2013-composite-prox-bundle,hare2015-prox-bundle-inexact-info,
monjezi2022-convg-prox-bundle-alg-nonsmooth-nonconvex}, 
quasi-Newton bundle \cite{luksan1998-bundle-newton-methd-nonsmooth-unconstr-min,maleknia2020-quasi-newton-bundle-grad-sampling}, 
and trust-region bundle methods \cite{monjezi2023-bundle-trust-region-locally-liptz-func}, etc.
In a proximal bundle method, it manages a bundle of subgradients (often obtained from previous iterations)  
to generate a convex piecewise linear local model, while a quadratic penalty term can be incorporated to control
the moving distance from the current point \cite{bagirov2014-bundle-method}. 
Recently the proximal bundle algorithm has been generalized with different assumptions on the objective
function \cite{hare2010-redistributed-prox-bundle,dao2016-nonconvex-bundle-delamination,
hare2015-prox-bundle-inexact-info,monjezi2022-convg-prox-bundle-alg-nonsmooth-nonconvex,
kiwie1996-restricted-step-prox-bundle}, 
which extended the application domain of this type of methods. 
Proximal bundle method has also been extended for handling nonsmooth, nonconvex constrained optimization
\cite{dao2015-bundle-method-non-non-constr,monjezi2019-infesb-prox-bundle-alg-nonsmooth-nonconvex-constr,
monjezi2021-filter-prox-bundle-non-non-constr,lv2017-prox-bundle-constr-non-non-inexact}.
In Quasi-Newton bundle methods \cite{mifflin1998-quasi-newton-bundle}, 
BFGS updating rules are used to approximate the inverse Hessian matrix
for deriving the search direction. In trust-region bundle methods \cite{monjezi2023-bundle-trust-region-locally-liptz-func}, 
the trust region management is used as an alternative to proximity control and 
it is a tighter control on the step size.

Gradient (subgradient) sampling (GS) methods 
\cite{lewis2005-grad-sampling-nonsmooth-nonconvex-opt,
curtis2015-quasi-newton-nonconvex-nonsmooth,
kiwiel2007-convg-grad-sampling-non-non-opt}
are in another category of effective methodology for nonsmooth optimization. 
The key difference between this type of methods and bundle methods 
is that the algorithm proactively draws samples in a neighborhood of the current point to access the subgradient
information nearby. As an advantage, the sampling technique can help collecting sufficient subgradient information
for improving the accuracy of the local model. However, it also takes additional computational time for sampling and selecting
useful subgradients in each iteration.  
Another side effect of a large sample size is that it leads to 
a high-complexity constrained quadratic program as a local model at every iteration, which is the most time-consuming step
in a contemporary GS method. Inexact subproblem solutions are proposed to mitigate this shortcoming 
\cite{curtis2021-grad-sampling-inexact-subproblem-solve}. 
The motivation of sampling technique is based on the assumption that 
the objective function can be non-differentiable at the current point. But if the point is in the domain of a smooth piece of the objective,
sampling is not necessary and it can save much computational effort. Since the algorithm is not informed by the macroscopic structural
information of the objective function, it will miss the opportunities of skipping the sampling procedure or reducing the sample size. 
As a consequence, it is difficult to establish a uniform rule of determining the an adaptive sample size  
based on the dimension of the problem and the performance of previous iterations. 
Gradient sampling methods are often coupled with quasi-Newton methods \cite{lewis2013-nonsmooth-opt-via-quasi-newton} 
for utilizing the neighborhood subgradient information in building a high-quality quadratic local model
\cite{curtis2015-quasi-newton-nonconvex-nonsmooth}.

An interesting research stream for nonsmooth nonconvex optimization in the recent decade 
is based on the assumption that the analytical form of the objective function is given and the
non-differentiability is only caused by operators of $\max()$, $\min()$, $|\cdot|$ and $(\cdot)^+$. 
Given this information,
techniques of sequential variable substitution and localizaiton have been developed to reformulate
the original unconstrained nonsmooth optimization problem into multiple constrained smooth optimization subproblems defined locally
\cite{griewank2013-stable-piecewise-linearization-gen-alg-diff,
griewank2016-first-sec-order-cond-piecewise-smooth,
fiege2019-alg-nonsmooth-succ-piecewise-linearization}.
A local subproblem is parameterized by a binary vector indicating the branch of each nonsmooth operator.
First and second order optimality conditions can be derived for the constrained local problem
\cite{griewank2016-first-sec-order-cond-piecewise-smooth}. Successive piecewise-linearization
algorithm has been developed for solving these local subproblems by exploiting the kink structure resulted 
by the branches of nonsmooth operators \cite{fiege2019-alg-nonsmooth-succ-piecewise-linearization,
griewank2016-lip-opt-graybox-piecewise-linearization}. The algorithm can 
be enhanced by the automatic differentiation technique in implementation \cite{fiege2017-alg-diff-piecewise-smooth-rob-opt}. 
Linear rate of convergence can be achieved using exact second-order information under the linear independent
kink qualification (LIKQ) \cite{griewank2019-relax-kink-qual-convg-rate-piecewise-smooth-func}.
There are a few shortcomings of this approach. First, adding auxiliary variables extends the dimension of 
an instance and increases complexity.
Second, it requires intensive reformulation process and artificial intervention,
which makes it difficult to implement for automatically transforming the original problem and solving nonsmooth
optimization problems at scale.

On the rate of convergence, Han and Lewis \cite{lewis2023-survey-descent} proposed a survey
descent algorithm that achieves a local linear rate of convergence for the family of max-of-smooth functions
with each component function being strongly convex. Mifflin et al. \cite{mifflin1998-quasi-newton-bundle}
has developed a quasi-Newton bundle method for nonsmooth convex optimization. The method consists
of a quasi-Newton outer iterations and using bundle method as inner iterations, and it has been shown
that the method achieves superlinear convergency in the outer iteration with certain assumptions and 
regularity conditions on the objective. However, the bundle method needs to take a sufficient amount of inner iterations
to find a solution with a certain accuracy level required by the outer iterations.
Charisopoulos and Davis \cite{charisopoulos2023-superlinear-convg-subgrad-sharp-semismooth} 
have recently developed an advanced subgradient method that achieves superlinear convergency 
for objective functions satisfying sharp growth and semi-smooth conditions
defined in the paper. The convergence rate of proximal bundle methods has
also been investigated in 
\cite{ruszczynski2017-convg-rate-bundle,diaz2023-opt-convg-rate-prox-bundle,
atenas2023-unified-analy-descent-bundle-convg-rate}.
We refer to the monograph \cite{cui-pang2021-modern-non-non-opt} for theory and algorithms
established for contemporary nonconvex nondifferentiable optimization.

In \tabref{tab:paper-compare}, we summarize the comparison of between related works and this paper
from a variety of perspectives.  

\begin{table}[p]
\floatpagestyle{empty}
\centering
\rotatebox{90}{
\begin{minipage}{\textheight}
\caption{Comparison of related works}
\label{tab:paper-compare}
\renewcommand{\arraystretch}{3}
\begin{tabular}{c|c|c|c|c|c}
\hline\hline
Literature & \makecell{Family of \\ nonsmooth functions} & \makecell{Requirement of \\ bundle points} & \makecell{Exploitation of \\ analytical structure} & \makecell{Parsing of \\ analytical expression} & \makecell{Convergence rate \\ analysis}  \\
\hline
\makecell{Griewank (2013) \cite{griewank2013-stable-piecewise-linearization-gen-alg-diff} \\
Fiege et. al. (2019) \cite{fiege2019-alg-nonsmooth-succ-piecewise-linearization} \\
Griewank et. al. (2016) \cite{griewank2016-lip-opt-graybox-piecewise-linearization}}
 & abs-smooth & \ding{55}  & \ding{51} & \makecell{case-by-case \\ reformulation \& lifting} & \ding{55}   \\
\hline
Kreimeier et. al. (2024) \cite{kreimeier2024-frank-wolfe-abs-smooth} & abs-smooth & \ding{55} & \ding{51} & \makecell{case-by-case \\ reformulation \& lifting} & \ding{55}  \\
\hline
\makecell{Griewank \& Walther  \\
(2016) \cite{griewank2016-first-sec-order-cond-piecewise-smooth}, 
(2019) \cite{griewank2019-relax-kink-qual-convg-rate-piecewise-smooth-func}}
& abs-smooth & \ding{55} & \ding{51} & \makecell{case-by-case \\ reformulation \& lifting} & \ding{51}  \\
\hline
\textbf{This paper} & piecewise smooth & \ding{51} & \ding{51}  & \makecell{automatic parsing during\\ function evalation } &  \ding{51} \\
\hline\hline
Literature & \makecell{Family of \\ nonsmooth functions} & \makecell{Bundle points \\ collection rule} & \makecell{Exploitation of \\ analytical structure} & \makecell{Using BFGS update} & \makecell{Convergence rate \\ analysis}  \\
\hline
\makecell{Burke et. al. (2005) \cite{lewis2005-grad-sampling-nonsmooth-nonconvex-opt} \\
Kiwiel (2007) \cite{kiwiel2007-convg-grad-sampling-non-non-opt} 
} & locally Lipschitz & random sampling & \ding{55} & \ding{51} & \ding{55}  \\
\hline
\makecell{
Lewis and Overton (2013) \cite{lewis2013-nonsmooth-opt-via-quasi-newton} \\
Curtis and Que (2015) \cite{curtis2015-quasi-newton-nonconvex-nonsmooth} \\
Curtis and Li (2021) \cite{curtis2021-grad-sampling-inexact-subproblem-solve} 
} & locally Lipschitz & random sampling & \ding{55} & \ding{51} & \ding{55}  \\
\hline
\makecell{
 Du and Ruszczy{\'n}ski (2017) \cite{ruszczynski2017-convg-rate-bundle} \\
 D{\'i}az and Grimmer (2023) \cite{diaz2023-opt-convg-rate-prox-bundle}
 Atenas et. al. (2023) \cite{atenas2023-unified-analy-descent-bundle-convg-rate}
} & locally Lipschitz & deterministic & \ding{55} & \ding{55} & \ding{51}  \\
\hline
\makecell{
 Han and Lewis (2023) \cite{lewis2023-survey-descent} \\
 Charisopoulos and Davis (2024) \cite{charisopoulos2023-superlinear-convg-subgrad-sharp-semismooth} 
} & locally Lipschitz & \ding{55} & \ding{55} & \ding{55} & \ding{51}  \\
\hline
\makecell{
Mahdavi-Amiri \& Yousefpour (2012) \cite{amiri2012-eff-nonsmooth-opt-loc-lip} 
\\
Gebken and Peitz (2021) \cite{gebken2021-eff-desc-methd-loc-lip}
}
 & locally Lipschitz & deterministic & \ding{55} & \ding{55} & \ding{55}  \\
\hline
\textbf{This paper}
 & piecewise smooth & \makecell{deterministic via \\ branch information} & \ding{51} & \ding{55} & \ding{51}  \\
\hline
\end{tabular}
\end{minipage}
}
\end{table}

\section{Encoding of piecewise-differentiable functions}\label{sec:prelim}
We begin by giving a formal definition of piecewise-differentiable functions that 
are analyzed in this paper. 
\begin{definition}\label{def:piecewise-diff}
A continuous function $f:\mR^n\to\mR$ is \textit{piecewise-differentiable} if there exists a finite set of functions 
$\{f_\theta:\;\theta\in\Theta\}$ satisfying the following conditions:
\begin{enumerate}
	\item The function $f_\theta: \tD_\theta\to\mR$ is continuously differentiable on the open domain $\tD_\theta\subseteq\mR^n$.
	\item For any $x\in\mR^n$, there exists a $\theta\in\Theta$ depending on $x$ such that $x\in\tD_\theta$ and $f(x) = f_\theta(x)$.
	\item For any $x\in\mR^n$, if $f_\theta(x)=f(x)$, then for any $\epsilon>0$, there exists an open set $C\subseteq\tD_\theta$ such that 
		$\overline{C}\subseteq\mB(x,\epsilon)$, $f(y)=f_\theta(y)\;\forall y\in{C}$ but
		$f(y)\neq{f}_{\theta^\prime}(y)\;\forall \theta^\prime\neq\theta\; \forall y\in{C}\cap\tD_{\theta^\prime}$, where $\mB(x,\epsilon)$ is the open ball with radius $\epsilon$ centered at $x$. 
	\item For any $x\in\mR^n$, there exists a small neighborhood $U$ of $x$ such that $\cM(y)\subseteq\cM(x),\;\forall{y}\in{U}$.
	\item If $f$ is differentiable at $x$, then $\nabla{f}(x)=\nabla{f}_\theta(x)$ for every $\theta$ satisfying $x\in\tD_\theta$ 
		 and $f(x)=f_\theta(x)$.
\end{enumerate}
\end{definition}
For a piecewise-differentiable function $f$, 
$\Theta$ is the set of \textit{branches}, $f_\theta$ is called the \textit{branch function} indexed by $\theta\in\Theta$,
$\tM(x):=\{\theta\in\Theta:\;x\in\tD_\theta\}$ is the set of \textit{feasible branches} at $x$,
and $\cM(x):=\{\theta\in\Theta: x\in\tD_\theta,\; f_\theta(x)=f(x)\}$ is set of \textit{active branches} at $x$.
The subset $\cD_\theta:=\{x:\theta\in\cM(x)\}$ is called the \textit{active domain} of the
branch $\theta$ (assuming $\cD_\theta\subseteq\tD_\theta$).
Note that one has $\cM(x)\subseteq\tM(x)$ by definition.
The \textit{multiplicity} of $f$ at $x\in\mR^n$ is the cardinality of $\cM(x)$.

To provide insights into \defref{def:piecewise-diff}, let us analyze its four five conditions:
Condition 1 ensures that each branch function is differentiable on its domain.
Condition 2 states that the value of $f$ at $x$ is determined by at least one branch function.
Condition 3 requires that every active branch at $x$ has a \textit{characteristic region} within 
any arbitrarily small neighborhood of $x$. Here, a \textit{characteristic region} of a branch $\theta$
is the set where $\theta$ is the only active branch. 
This condition plays a crucial role in preventing branch redundancy. Specifically, if a branch $\theta\in\cM(x)$
violates this condition, there must exist an arbitrarily small open neighborhood $C$ of $x$ such that for
every $y\in{C}$, some other branch $\theta^\prime$ (depending on $y$) satisfies $f_\theta(y)=f_{\theta^\prime}(y)$.
Intuitively, this means $\theta$ is redundant near $x$, and its domain can be refined--by redefining
$\tD_\theta\gets\tD_\theta\setminus\overline{C}$--without altering the function’s behavior.
Condition 4 ensures that the set of active branches in a sufficiently small neighborhood of $x$ 
can only be a subset of $\cM(x)$. Without this condition, it is possible to construct a branch function $f_\theta$
with domain $\tD_\theta$ in the way that $x$ is at the boundary of $\tD_\theta$ but the function value 
satisfies $\lim_{y\to{x}}f_\theta(y)=f(x)$ and there exists a sequence $\{x_n\}$ satisfying $\theta\in\cM(x_n)$. 
In this case, the branch $\theta$ has active points in any sufficiently small neighborhood of $x$, but $x$ as the limit
point of $\{x_n\}$ is not in $\tD_\theta$. This ill situation can be resolved by adjusting the definition of $f_\theta$ and $\tD_\theta$.
Condition 5 states that if $x$ is at a characteristic region of a branch $\theta$, i.e., $\cM(x)=\{\theta\}$,
then $\nabla{f}(x)=\nabla{f}_\theta(x)$.


\begin{definition}[encodablility]\label{def:encodablility}
A piecewise-differentiable function $f$ is \textit{encodable} if for any $x\in\mR^n$, the 
set $\cM(x)$ of active branches is accessible and the gradient $\nabla{f}_\theta(x)$ is accessible
for any $\theta\in\cM(x)$.
\end{definition}
We emphasize that encodability constitutes an algorithmic property of piecewise-differentiable functions, 
which relies on having access to $f$'s analytical expression or numerical representation. 
This property enables exploitation of $f$'s analytical structure. 
Crucially, determining $\cM(x)$ requires knowledge of the branching rules for each nonsmooth operator in $f$. 
In general, a piecewise-differentiable function cannot be encoded if only its function values are accessible.

\begin{definition}[Clarke differential]
Given a function $f:\mR^n\to\mR$, for every $x$, the Clarke differential $\partial{f}(x)$ is defined as:
\beq\label{def:clarke-diff}
\partial{f(x)}:=\conv(\{s\in\mR^n:\exists\{x_i\}^\infty_{i=1}\;\textrm{s.t. }x_i\to{x},\;\nabla{f(x_i)}\to{s}\}).
\eeq 
\end{definition}

\begin{assumption}
Assume that each branch function $f_\theta$ of the encodable function $(f,\cM)$
is $L_\Theta$-smooth for a uniform Lipschitz coefficient $L_\Theta$. 
\end{assumption}

Given an encodable function $(f,\cM)$, a point $x$ and a nonempty subset $\cC$ of $\Theta$,
we will frequently refer to the following quadratic program:
\begin{equation}
\min_{\lambda}\;\norm{\sum_{\theta\in\cC}\lambda_\theta\nabla f_\theta(x)}^2 \quad \textrm{s.t. }\;\sum_{\theta\in\cC}\lambda_\theta=1,\quad\lambda_\theta\ge0. \label{opt:QP-x-C}  \tag{\textrm{QP-$x$-$\mathcal{C}$}}
\end{equation}
Denote $\sol{\QPt{x}{\cC}}$ as the optimal solution $\lambda^*$ of $(\QPt{x}{\cC})$, 
and $\vc{\QPt{x}{\cC}}$ as the optimal vector $d^*=\sum_{\theta\in\cC}\lambda^*_\theta\nabla f_\theta(x)$. 
This vector is referred as the \textit{joint gradient} of branch functions in $\cC$ at the point $x$.
Joint gradients are intensively used in \secref{sec:CID-JGD} to develop an algorithm that can achieve
linear convergence. A similar notion has been used in the gradient sampling method 
\cite{lewis2005-grad-sampling-nonsmooth-nonconvex-opt} for optimizing locally-Lipschitz
objectives. The difference is that in \cite{lewis2005-grad-sampling-nonsmooth-nonconvex-opt}, 
the gradients are taken at the sampled points.

\begin{table}
	\centering
	\begin{tabular}{cl}
		\hline\hline
		Notation & Definition \\
		\hline
		 $\Theta$ & the set of codes corresponding to all branch functions of $f$ \\
		 $\cM$ & the multi-valued mapping from $\mR^n$ to $\Theta$ (\defref{def:encodablility}) \\
		 $\tD_\theta$ & domain of the branch function $f_\theta$ for $\theta\in\Theta$ (\defref{def:encodablility}) \\
		 $\cD_\theta$ & the active domain of the branch function $f_\theta$ for $\theta\in\Theta$ (\defref{def:encodablility}) \\
		 $\tM(x)$ & $\theta\in\tM(x)\Longleftrightarrow x\in\tD_\theta$ (\defref{def:encodablility}) \\
		 $L_\Theta$ & the common Lipschitz constant of $\nabla f_\theta$ for all $\theta\in\Theta$ \\
		 $\mB(x,r)$ & the open ball of radius $r$ at the center $x$ \\
		 $G_f$ & an upper bound of $\norm{\nabla{f}_\theta(x)}$ for all $x\in\X$ and $\theta\in\Theta$, \\
		 	& where $X$ is a sufficiently large closed set that contains all points of interest \\
		\hline
		$\mudec$ & decreasing ratio of the trust-region size and trust-step size 
		(\algoref{alg:trust-region-nonconvex}) \\
		$\rho_0$ & the threshold of decrease rate for a successful move (\algoref{alg:trust-region-nonconvex} and \algoref{alg:global_and_local_converg}) \\
		$\Thetadisc$ & a dynamic set of branch codes that has been discovered so far in \algoref{alg:trust-region-nonconvex} \\
		$\gamma$ & an internal parameter used in \algoref{alg:global_and_local_converg} for measuring gap reduction \\
		$z_\theta$ & the representative point corresponding to the branch $\theta$ for $\theta\in\Thetadisc$ (\algoref{alg:trust-region-nonconvex})\\
		$\cT(x,r)$ & defined as $\{\theta\in\Thetadisc:\norm{x-z_\theta}\le r\}$ (\algoref{alg:trust-region-nonconvex}) \\
		\hline
	\end{tabular}
	\caption{List of notations used in the paper}\label{tab:notations}
\end{table}

In the following, we provide two examples to illustrate the notions defined above and the motivation
of functional encoding.
\begin{example}\label{exp:basic-concept}
Consider the function $f(x)=\max\big\{-x+1,\;\frac{1}{4}x,\;x-6\big\}$. It has a non-smooth operator $\max(\cdot)$,
which splits $f$ into three branch functions: $f_1(x):=-x+1$, $f_2(x)=\frac{1}{4}x$ and $f_3(x)=x-6$
with the branch code $\theta=1,2,3$, respectively. Consider five points $x_1=0$, $x_2=\frac{4}{5}$,
$x_3=4$, $x_4=8$ and $x_5=10$. The $x_1$ makes $f_1$ active and $x_2$ makes both $f_1,f_2$ active, etc.
Based on \defref{def:encodablility}, the multi-valued mapping $\cM$ maps $x_i$ ($i=1,2,3,4,5$) to the following
set values:
\bdm
\cM(x_1)=\{1\},\;\cM(x_2)=\{1,2\},\;\cM(x_3)=\{2\},\;\cM(x_4)=\{2,3\},\;\cM(x_5)=\{3\}.
\edm
The active domain $\cD_\theta$ is given by
\bdm
\cD_1=\Big(-\infty,\frac{4}{5}\Big],\;\cD_2=\Big[\frac{4}{5},8\Big],\;\cD_3=\Big[8,\infty\Big).
\edm
\end{example}

\begin{example}[\cite{haarala2004-mem-bundle-nonsmooth-opt}]\label{exp:branching}
Chained Crescent II (enhanced)
\beq
f(x)=\sum^{n-1}_{i=1}\max\{x^2_i+(x_{i+1}-1)^2+|x_{i+1}-1|,-x^2_i-(x_{i+1}-1)^2+|x_{i+1}+1|\}.
\eeq
The function contains $3(n-1)$ non-smooth operators: the $\max\{\cdot\}$,
the absolute operator $|\cdot|$ of $|x_{i+1}-1|$ and the absolute operator $|\cdot|$ 
of $|x_{i+1}+1|$ in each term in the summation. The function can be encoded as 
\beq
\theta=\bigoplus^{n-1}_{i=1}(\theta^i_1,\theta^i_2,\theta^i_3),
\eeq
where $\theta^i_1,\theta^i_2,\theta^i_3\in\{1,2\}$ respectively encodes the branch of the three operators
in the $i$-th term of the summation. \figref{fig:branching} illustrates the functional encoding and branching
scheme of the three non-smooth operators in the $i$-th term of the summation. In particular, the code 
$(\theta^i_1,\theta^i_2,\theta^i_3)=(1,2,1)$ corresponds to the resulting term $x^2_i+(x_{i+1}-1)^2-x_{i+1}+1$,
while the code $(\theta^i_1,\theta^i_2,\theta^i_3)=(1,2,2)$ corresponds to the same resulting term, but
they have different active domains. Once the values of $(\theta^i_1,\theta^i_2,\theta^i_3)$ are specified 
for all $i\in\{1,\ldots,n-1\}$, we obtain a specific branch function corresponding to the code $\theta$.
\end{example}
A key observation from the example above is that determining the active branch codes of $f$ at a point
$x$ is a natural by-product of function evaluation, since computing $f(x)$ already requires identifying the
active branch of each nonsmooth operator. Therefore, functional encoding does not introduce additional
computational overhead. 

\begin{figure}
\centering
    \includegraphics[width=1.0\linewidth,trim=1cm 18.5cm 0.5cm 5cm]{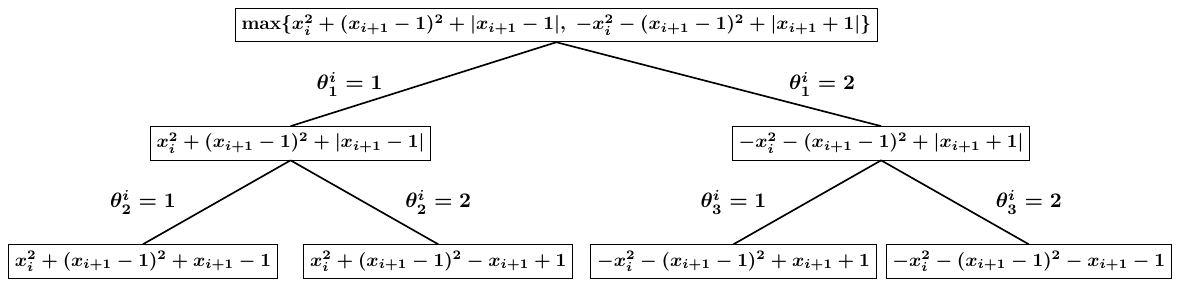}
    \caption{Non-smooth operator branching scheme of the function in \expref{exp:branching}.}
    \label{fig:branching}
\end{figure} 

\begin{example}
Consider a piecewise-differentiable function $f$ defined using the following branching rule:
\begin{equation}
f(x) =\begin{cases}
	-0.5x^2 + 2x &\text{if } x\in(0,2], \\
	0.5x^2 - 4x + 8 &\text{if } x\in(2,4], \\
	1.5x - 6 &\text{if } x\in(4,6).
	\end{cases}
\end{equation}
See \figref{fig:rule-based-piecewise-func} for the illustration. 
Let $\Theta=\{1,2,3\}$. Let the three branch functions be $f_1(x)=-0.5x^2 + 2x$, 
$f_2(x)=0.5x^2 - 4x + 8$ and $f_3(x)=1.5x - 6$. Their domains are  
$\tD_1=(0,3)$, $\tD_2=(1,5)$ and $\tD_3=(3,6)$, respectively. This definition 
is not unique. Their active domains
are $\cD_1=(0,2]$, $\cD_2=[2,4]$ and $\cD_3=[4,6)$, respectively. 
In this example, $f$ is not an abs-smooth function as in the previous two examples, but it is
still encodable by \defref{def:piecewise-diff}. This example points out that the family
of nonsmooth functions investigated in this work is broader than abs-smooth functions
handled in \cite{griewank2016-first-sec-order-cond-piecewise-smooth,griewank2019-relax-kink-qual-convg-rate-piecewise-smooth-func}. 
\end{example}
\begin{figure}
\centering
    \includegraphics[scale=1.2]{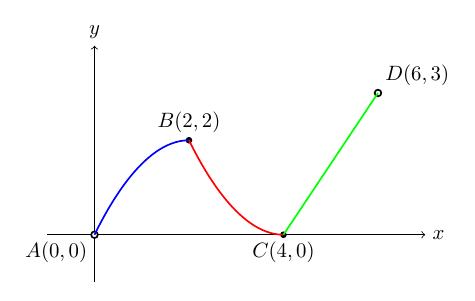}
    \caption{A rule-based piecewise-differentiable function that is encodable.}
    \label{fig:rule-based-piecewise-func}
\end{figure}

Our methods are particularly effective for problem instances where the number of branches 
near local minimizers remains relatively small--specifically, 
cases where the branch count does not grow exponentially with problem dimension. 
For example, consider objectives of the form
$f(x)=g(x) + \norm{x}_1$, where the $\ell_1$-norm $\norm{x}_1$ induces $2^n$ branches (for dimension $n$)
at the origin, while $g(\cdot)$ has only a constant number of branches (independent of $n$).
In such cases, the $\ell_1$-norm term can be moved into the subproblem that computes 
the search direction (as developed in this paper) and handled via standard techniques.

For the case that the multiplicity at $x$ exponentially depends on $n$, the practical implementation
of our methods only need to access an arbitrary unvisited branch from $\cM(x)$ rather than building the entire set $\cM(x)$.
Furthermore, while the issue of exponentially many branch functions near minimizers poses a challenge 
for all nonsmooth optimization methods, it manifests differently across approaches. 
In conventional methods, this challenge remains implicit, whereas our method explicitly addresses it. 
Crucially, our computational results show that even in such demanding cases--where the number 
of branch functions grows exponentially near minimizers--our method consistently outperforms competing approaches.

\begin{proposition}
For a piecewise-differentiable function $f$, the Clarke differential has the following simplified representation:
\beq
\partial f(x) = \conv(\{\nabla f_\theta(x): \theta\in\cM(x)\}).
\eeq
\end{proposition}
\begin{proof}
Based on the definition of $\partial{f}$ in \eqref{def:clarke-diff} and the condition 3 in \defref{def:piecewise-diff}, 
we conclude that $\conv(\{\nabla f_\theta(x): \theta\in\cM(x)\})\subseteq\partial f(x)$.
It suffices to prove $\partial f(x)\subseteq\conv(\{\nabla f_\theta(x): \theta\in\cM(x)\})$.
Consider any sequence $x_i\to{x}$ and $f$ is differentiable at $x_i$ with the limit
$\nabla{f}(x_i)\to{s}$. Since $\Theta$ is finite, there exists a branch $\theta$ such that the event
$\theta\in\cM(x_i)$ occurs infinite time as $i\to\infty$. Condition 3 in \defref{def:piecewise-diff}
implies $\theta\in\cM(x)$.
Then it is possible to find a subsequence
$\{i_k\}^\infty_{k=1}$ such that $\theta\in\cM(x_{i_k})$. Using the condition 5 in \defref{def:piecewise-diff}
and the continuity of $\nabla{f}_\theta$, we conclude that the following limits hold:
\bdm
\ba
s=\lim_k\nabla{f}(x_{i_k}) = \lim_k\nabla f_\theta(x_{i_k}) = \nabla f_\theta(x).
\ea
\edm
This proves $\partial f(x)\subseteq\conv(\{\nabla f_\theta(x): \theta\in\cM(x)\})$.
\end{proof}

\begin{corollary}
Given a piecewise-smooth function $f$, $x$ is a Clarke stationary point of $f$ if and only if
if $\bs{0}\in\conv(\{\nabla f_\theta(x): \theta\in\cM(x)\})$.
\end{corollary}
The algorithmic insight of the above corollary is that the validation
of Clarke stationarity can be restricted to only concerning the set of active branch functions
at a point. 

\section{A gradient descent method with branch information}\label{sec:CTI-TR}
We developed a branch-information-driven gradient descent method (BIGD) 
to optimize an encodable piecewise-differentiable objective function.
The method is presented in \algoref{alg:trust-region-nonconvex}.
The key innovation of our approach lies in dynamically maintaining a set $\Thetadisc$ of branches that have been visited
by the algorithm and a corresponding set $\cZ^n=\{z^n_\theta: \theta\in\Thetadisc\}$ of representative points in the
active domain of these branches. Here, each $z^n_\theta\in\cD_\theta$ serves as the unique representative
point for branch $\theta$ at iteration $n$, which may be updated in subsequent iterations.
During optimization, $\Thetadisc$ and $\cZ^n$ are leveraged to generate a candidate set of 
descent directions. These directions are then evaluated by a customized line-search method within our 
framework to select the most promising descent direction and step size at each iteration $n$.

The algorithm consists of outer and inner iterations. The outer iteration (indexed by $n$ in the algorithm) 
generates the sequence $\{x_n\}$ of points which converges to a Clarke stationary point. All inner iterations 
(indexed by $k$ in the algorithm) within each outer iteration are designed to test the efficiency of candidate descent directions. 
There are three functions involved in an inner iteration:
\Call{BranchSelection}{}, \Call{GradientComputation}{} and \Call{LineSearch}{}. The function \Call{BranchSelection}{}
selects a subset $\cT^n_k$ of branches based on the distance from $x^n$ to their representative points $z^n_\theta$. 
The function \Call{GradientComputation}{} computes a candidate gradient ascent direction $d^n_k$ (or a gradient descent direction $-d^n_k$)
by solving the quadratic program \eqref{opt:qp-sum-lambda-grad-f} that depends on $\cT^n_k$ and $\cZ^n$.
After that, the \Call{LineSearch}{} function computes a step size along $-d^n_k$. Within an outer iteration,
the parameter $\eta$ represents the largest decrease of the objective value in the inner iterations. 
The $\alpha$ and $d$ represent the step size and gradient  direction associate to $\eta$. 
Note that all the trial points generated in the line search are also used to update $\Thetadisc$ and $\cZ^n$. 

\begin{figure}[htbp]
    \centering
    \begin{subfigure}{0.49\textwidth}
        \centering
        \includegraphics[width=\linewidth]{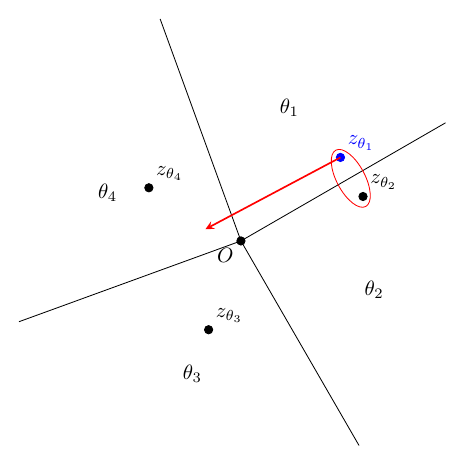}
        \caption{}
        \label{fig:a}
    \end{subfigure}
    \hfill
    \begin{subfigure}{0.49\textwidth}
        \centering
        \includegraphics[width=\linewidth]{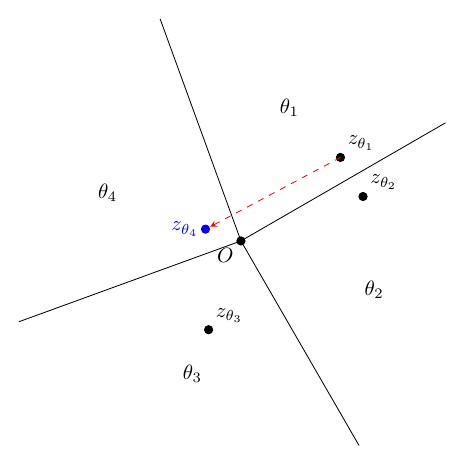}
        \caption{}
        \label{fig:b}
    \end{subfigure}
    \caption{An illustration of \algoref{alg:trust-region-nonconvex}:
    Figure (a) shows that there are four branches $\theta_1$, $\theta_2$, $\theta_3$ and $\theta_4$ near a local minimizer $O$.
    Their representative points are $z_{\theta_1}$, $z_{\theta_2}$, $z_{\theta_3}$ and $z_{\theta_4}$, respectively, which are visited 
    by the algorithm. Suppose $x=z_{\theta_1}$ is the current point, and branches $\theta_1$ and $\theta_2$
    are selected to compute the descent direction. After the line search along this direction (red arrow), 
    the point moves from $z_{\theta_1}$ to a point in $\tD_{\theta_4}$. Figure (b) shows that the representative point of $\theta_4$
    has been updated after the move. }
    \label{fig:both}
\end{figure}

\begin{algorithm}
{\footnotesize
	\caption{\scriptsize A branch-information-driven gradient descent method (BIGD).} \label{alg:trust-region-nonconvex}
	\begin{algorithmic}[1]
			\State{Initialization: $x^1\gets{x}_\textrm{init}$, $n\gets 1$ and $\Thetadisc\gets\cM(x^1)$.}
			\While{$n<\infty$} \label{lin:outer-iter}
				\State{Let $\cZ^n=\cup_{\theta\in\cT(x^n,r_0)}\{z_\theta\}$. (See \tabref{tab:notations} for the definition of $\cT(\cdot,\cdot)$.) }
				\State{Set $k\gets1$ and $\eta\gets0$.}
				\While{$k\le|\cZ^n|$} \label{lin:update_k}
					\State{Let $\cT^n_k,r=\Call{BranchSelection}{\cZ^n,x^n,k}$}
					\State{Let $d^n_k=\Call{GradientComputation}{\cT^n_k,\cZ^n}$.}
					\State{\textbf{If }$\norm{d^n_k}=0$, \textbf{continue}.}
					\State{$\beta,\eta,signal=\Call{LineSearch}{x^n,d^n_k,r,\eta,k}$.} \label{lin:call-line-search}
					\If{$signal=1$}
						\State{Let $\alpha\gets\beta$, $d\gets{d^n_k}$.}
					\EndIf
					\State{Set $k\gets k+1$.}
				\EndWhile
			\State{Let $\alpha_n\gets\alpha$, $d^n\gets{d}$ and $x^{n+1}=x^n-\alpha_nd^n$.} \label{lin:x^n_update}
			\State{\Call{TerminationCheck}{$x^{n+1}$}.}
			\State{Set $n\gets{n+1}$.}
			\EndWhile
	\end{algorithmic}
	\vspace{-5pt}\noindent\makebox[\linewidth]{\rule{\linewidth}{0.4pt}}
	\begin{algorithmic}[1]
		\Function{BranchSelection}{$\cZ,x,k$}
			\State{Within the set $\cZ$, select the $k$ distinct nearest points $z_1,\ldots,z_k$
					 from $x$, breaking tie arbitrarily.} \label{lin:k-nearest}
			\State{Let $\cT=\cup^k_{i=1}\cM(z_i)$ and let $r\gets\norm{z_k-x}$.} \label{lin:r_k}
			\State{\textbf{return} $\cT, r$.}
		\EndFunction
	\end{algorithmic}
	\vspace{-5pt}\noindent\makebox[\linewidth]{\rule{\linewidth}{0.4pt}}
	\begin{algorithmic}[1]
		\Function{GradientComputation}{$\cT,\cZ$}
					\State{Solve the following quadratic program:}
					\beq\label{opt:qp-sum-lambda-grad-f}
					\min_\lambda \norm{\sum_{\theta\in\cT}\lambda_\theta\nabla f_\theta(z_\theta)}^2
					\;\textrm{ s.t. } \sum_{\theta\in\cT}\lambda_\theta = 1,\;\lambda_\theta\ge 0\;\forall\theta\in\cT,
					\eeq
					\State{where $z_\theta$ is fetched from $\cZ$ for the branch $\theta$.
					Let $\lambda^*$ be the optimal solution of the above problem, and let 
					$d^*=\sum_{\theta\in\cT}\lambda^*_\theta\nabla f_\theta(z_\theta)$.}
				\State{\textbf{return} $d^*$.}
		 \EndFunction
	\end{algorithmic}
	\vspace{-5pt}\noindent\makebox[\linewidth]{\rule{\linewidth}{0.4pt}}
	\begin{algorithmic}[1]
		\Function{LineSearch}{$x,d,r,\eta,k$}
			\State{Parameters: $r_0>0$, $\alphamax>0$, $\rho_0\in(0,1)$ and $\mudec\in(0,1)$.}
			\State{Set $t\gets 1$, $\alpha_{t}\gets\min\{\alphamax,r_0/\norm{d}\}$, $flag\gets 0$ and $signal\gets 0$.}
					\While{$flag=0$}\label{lin:while-loop-flag}
						\State{Compute the following ratio: $\rho = [f(x) - f(x-\alpha_{t}d)]/(\alpha_{t}\norm{d}^2)$.}
						\If{$\rho\ge\rho_0$} \label{lin:rho^n_k>=rho_0}
							\If{$f(x) - f(x-\alpha_{t}d)>\eta$}
								\State{Set $\eta\gets{f}(x) - f(x-\alpha_{t}d)$.} \label{lin:x(n+1)=x(n)-ad}
								\State{Set $signal\gets1$.} \label{lin:flag_1}
							\EndIf
							\State{Set $flag\gets 1$.}
							\State{\Call{BranchPointUpdate}{$x-\alpha_td$, $x-\alpha_td$}.}\label{lin:update1}
						\Else
							\State{\Call{BranchPointUpdate}{$x$, $x-\alpha_t d$}.}\label{lin:update2}
							\If{$k=1$ or ($k\ge 2$ and $\norm{\alpha_t d}>r$)} \label{lin:k=1_or_k>=2}
								\State{Set $\alpha_{t+1}\gets\mudec\alpha_t$ and $t\gets t+1$.}\label{lin:alpha-dec}
							\Else
								\State{Set $flag\gets 1$.} \label{lin:flag_2}
							\EndIf
						\EndIf
					\EndWhile	
			\State{\textbf{return} $\alpha_t,\eta,signal$.}
		\EndFunction
	\end{algorithmic}
	\vspace{-5pt}\noindent\makebox[\linewidth]{\rule{\linewidth}{0.4pt}}
	\begin{algorithmic}[1]
		\Procedure{TerminationCheck}{$x$}
			\State{Let $d=\Call{GradientComputation}{\cM(x)}$.}
			\If{$\norm{d}=0$}
				\State{Stop the program and claim that $x$ is Clarke stationary.}
			\Else
				\State{\Call{BranchPointUpdate}{$x, x$}.}
			\EndIf
		\EndProcedure
	\end{algorithmic}
	\vspace{-5pt}\noindent\makebox[\linewidth]{\rule{\linewidth}{0.4pt}}
	\begin{algorithmic}[1]
		\Procedure{BranchPointUpdate}{$x$, $\xtr$}
			\For{$\theta\in\cM(\xtr)$}
				\If{there does not exist $\theta$ in $\Thetadisc$}
					\State{Set $\Thetadisc\gets\Thetadisc\cup\{\theta\}$ and $z_\theta\gets\xtr$.}
				\ElsIf{$\norm{\xtr-x}<\norm{z_\theta-x}$}
					\State{Set $z_\theta\gets\xtr$.}
				\EndIf
			\EndFor
		\EndProcedure
	\end{algorithmic}
}
\end{algorithm}

\thmref{thm:convg-stationary-pt} is the main result of the section. 
It shows that the sequence $\{x_n\}$ generated by the algorithm converges to a Clarke stationary point. 
We first present a few technical results that are used in proving this theorem.
The following proposition is a standard result of Taylor expansion for a same branch function.
The result is the foundation for the step size analysis.
\begin{proposition}\label{prop:joint-taylor}
For any point $x$, any subset of branch functions $\cC\subseteq\tM(x)$,
let $d=\vc{\QPt{x}{\cC}}$. Then, for all $\theta\in\cC$ and step size $\alpha$ satisfying $x-\alpha d\in\tD_\theta$,
the following inequality holds:
\beq
f_\theta(x) - f_\theta(x-\alpha d) 
\ge \alpha\norm{d}^2 - \frac{1}{2}L_\Theta\alpha^2\norm{d}^2,
\eeq   
where $L_\Theta$ is the Lipschitz constant associated with the branch functions.
\end{proposition}
\begin{proof}
This is a direct application of the fact $\inner{\nabla f_\theta(x)}{d}\ge\norm{d}^2$ for any $\theta\in\cC$
from \propref{prop:<u,vi>}, and the Taylor expansion and $L_\Theta$-smoothness.
\end{proof}

The following proposition concerns the ratio between $f(x) - f(x-\alpha{d})$ and $\alpha\norm{d}^2$.
Note that the active branch at $x-\alpha{d}$ may be different from the one at $x$. This subtlety
needs to be considered in the analysis.
\begin{proposition}\label{prop:(f-f)/(ad^2)}
Let $\cC$ be a non-empty subset of $\tM(x)$ at any $x\in\mR^n$
and let $d=\qpvc{x}{\cC}$ satisfying $\norm{d}>0$. 
Suppose $\theta_1\in\cM(x)\cap\cC$ and $\theta_2\in\cM(x-\alpha{d})$
for a step size $\alpha$.
If $\theta_2\in\cC$ and $f_{\theta_1}(x)\ge{f}_{\theta_2}(x)$, 
the following inequality holds:
\beq
\frac{f(x) - f(x-\alpha{d})}{\alpha\norm{d}^2} \ge 1-\frac{1}{2}\alpha{L_\Theta}.
\eeq
\end{proposition}
\begin{proof}
By definition we have
\bdm
\ba
&\frac{f(x) - f(x-\alpha{d})}{\alpha\norm{d}^2} = \frac{f_{\theta_1}(x) - f_{\theta_2}(x-\alpha{d})}{\alpha\norm{d}^2} \\
&\ge\frac{f_{\theta_2}(x) - f_{\theta_2}(x-\alpha{d})}{\alpha\norm{d}^2}
\ge1-\frac{1}{2}\alpha{L_\Theta},
\ea
\edm
where \propref{prop:joint-taylor} is applied to derive the last inequality.
\end{proof}

Our convergence analysis points out that just using the current active branch functions to compute
a gradient ascent direction at each $x^n$ can not guarantee a lower-bounded step size in the line search. 
As a consequence, this approach cannot make $\{x^n\}$ converge to a Clarke stationary point. 
Instead, we should take into account branches with $z^n_\theta$ that are very close to $x^n$ but are not
active at $x^n$, i.e., the branches with the representative point of visit in some small neighborhood of $x^n$.
This motivates us to consider the quadratic program \eqref{opt:min_lambda_z}. The following lemma
quantifies the quality of its solution when $z_\theta$'s are very close to the point $x$. 
The result of this lemma is directly used in analyzing the step size.  
\begin{lemma}\label{lem:nonconvex-big-step}
Let $(f,\cM)$ be an encodable function, $x$ be a point with $\norm{\qpvc{x}{\cM(x)}}>0$. 
Let $\cC$ be a nonempty subset of $\cM(x)$ and $z_\theta\in\cD_\theta$ for all $\theta\in\cC$.
Let $\lambda$ be the optimal solution of the following quadratic program
\beq\label{opt:min_lambda_z}
\min_\lambda \norm{\sum_{\theta\in\cC}\lambda_\theta\nabla f_\theta(z_\theta)}^2\;
\mathrm{ s.t. } \sum_{\theta\in\cC}\lambda_\theta=1,\;\lambda_\theta\ge 0\;\forall\theta\in\cC,
\eeq
and $d=\sum_{\theta\in\cC}\lambda_\theta\nabla f_\theta(z_\theta)$. 
Then for any constant $0<\rho<1$, there exist constants $\alpha_0,\epsilon_0>0$ such that
the following inequality holds if $\alpha\le\alpha_0$ and $\max_{\theta\in\cC}\norm{z_\theta-x}<\epsilon_0$:
\beq
\frac{f(x)-f_\theta(x-\alpha d)}{\alpha\norm{d}^2}\ge \rho \quad \forall \theta\in\cC.
\eeq
\end{lemma}
\begin{proof}
Let $d_0=\qpvc{x}{\cC}$. Note that $\norm{d_0}\ge\norm{\qpvc{x}{\cM(x)}}>0$.  
Consider any $\theta\in\cC$. Using the Taylor expansion and Lipschitz condition 
we can derive 
\beq\label{eqn:f-f>a*d^2}
\ba
&f(x)-f_\theta(x-\alpha d)\ge\alpha\inner{\nabla f_\theta(x)}{d} - \frac{1}{2}\alpha^2L_\Theta\norm{d}^2 \\
&\ge\alpha\inner{\nabla f_\theta(z_\theta)}{d} - \alpha\norm{\nabla f_\theta(x) - \nabla f_\theta(z_\theta)}\cdot\norm{d} - \frac{1} {2}\alpha^2L_\Theta\norm{d}^2 \\
&\ge\alpha\inner{\nabla f_\theta(z_\theta)}{d} - \alpha{L_\Theta}\norm{x-z_\theta}\cdot\norm{d} - \frac{1} {2}\alpha^2L_\Theta\norm{d}^2 \\
&\ge\alpha\norm{d}^2 - \alpha{L_\Theta}\norm{x-z_\theta}\cdot\norm{d} - \frac{1} {2}\alpha^2L_\Theta\norm{d}^2,
\ea
\eeq
where the last inequality follows from \propref{prop:<u,vi>}.
By \propref{prop:d_limit}, we have 
\beq\label{eqn:lim_d_to_d0}
\lim_{\substack{z_{\theta^\prime}\to x\\\forall\theta^\prime\in\cC}} \norm{d}= \norm{d_0}.
\eeq
Using \eqref{eqn:lim_d_to_d0} in \eqref{eqn:f-f>a*d^2}, we conclude that
\beq
\lim_{\substack{z_{\theta^\prime}\to x\\\forall\theta^\prime\in\cC}} f(x)-f_\theta(x-\alpha d)\ge
\alpha\norm{d_0}^2 - \frac{1} {2}\alpha^2L_\Theta\norm{d_0}^2.
\eeq
Therefore, we have
\beq
\ba
&\lim_{\substack{z_{\theta^\prime}\to x\\ \forall\theta^\prime\in\cC}}\frac{f(x)-f_\theta(x-\alpha d)}{\alpha\norm{d}^2}
\ge\frac{\alpha\norm{d_0}^2 - \frac{1} {2}\alpha^2L_\Theta\norm{d_0}^2}{\alpha\norm{d_0}^2} 
=1-\frac{1}{2}\alpha L_\Theta. 
\ea
\eeq
Therefore, the following inequality holds for some small enough $\delta>0$
if $\alpha$ and $\epsilon_0$ are sufficiently small
and $\norm{z_{\theta^\prime}-x}<\epsilon_0\;\forall\theta^\prime\in\cC$:
\beq
\ba
\frac{f(x)-f_\theta(x-\alpha d)}{\alpha\norm{d}^2}\ge 1-\frac{1}{2}\alpha{L_\Theta}-\delta>\rho,
\ea
\eeq
which concludes the proof.
\end{proof}

\begin{theorem}\label{thm:convg-stationary-pt}
For an encodable function $(f,\cM)$, suppose the \algoref{alg:trust-region-nonconvex} outputs a sequence 
$\{x^n\}^\infty_{n=1}$ that is bounded. Then the sequence converges to a Clarke stationary point of $f$. 
\end{theorem}
\begin{proof}
Claim 1. The sequence generated by the algorithm satisfies $f(x^n)>f(x^{n+1})$ as long as 
$x^n$ is not a Clarke stationary point. 

Proof of Claim 1. We first observe that $z_1=x^n$ at \linref{lin:k-nearest} of the algorithm. 
This is because the procedure \textproc{BranchPointUpdate}($x^n, x^n$) has set $z_\theta=x^n$
for every $\theta\in\cM(x^n)$, and hence $\cM(x^n)\subseteq\cT(x^n,r_0)$ and $x^n\in\cZ$ by 
definition of these sets. Consider the while-loop at \linref{lin:update_k} for $k=1$ in the main algorithm.
One has $\cT^n_1=\cM(x^n)$. Since $x^n$ is not stationary, one has $\norm{d^n_1}>0$.  
By Condition 4 of \defref{def:piecewise-diff}, if $\alpha_t$ is sufficiently small, 
one has $\cM(x^n-\alpha_td^n_1)\subseteq\cM(x^n)$. Applying \propref{prop:(f-f)/(ad^2)}
with $\cC=\cM(x^n)$ and $\theta_1=\theta_2\in\cC$ in the proposition, the ratio $\rho$ should satisfy
\bdm
\rho = \frac{f(x^n) - f(x^n-\alpha_{t}d^n_1)}{\alpha_{t}\norm{d^n_1}^2}\ge1-\frac{1}{2}\alpha_tL_\Theta\ge\rho_0,
\edm
given that $\alpha_t$ is small enough. It shows that the condition at \linref{lin:rho^n_k>=rho_0}
of \Call{LineSearch}{} called at inner iteration $k=1$ should hold at some $t$, and also
 $f(x^n)-f(x^{n+1})\ge\rho_0{\alpha_t}\norm{d^n_1}^2>0$.
This proves the claim.

Claim 2. The while-loop in the \Call{LineSearch}{} function terminates in finitely many iterations. 

Proof of Claim 2. The while-loop terminates when $flag=1$. Consider \linref{lin:k=1_or_k>=2}
in \Call{LineSearch}{}. For $k\ge2$, the $\alpha_t$ needs
to satisfy $\alpha_t>r/\norm{d}$ to be qualified for further reduction. When $\alpha_t$ becomes
smaller than this threshold, $flag$ will be set to 1 and the while-loop terminates.
For $k=1$, the proof of Claim 1 has shown that $\rho\ge\rho_0$ will satisfy if $\alpha_t$
becomes small enough, and hence the while-loop also terminates in this case.   

Claim 3. The sequence generated by the algorithm can only have exactly one limit point. 

Proof of Claim 3. Suppose the sequence has more than one limit points.  
If the algorithm terminates, a Clarke stationary point has been identified and 
the claim obviously holds. We therefore assume that it does not terminate. 
Let $y$ and $y^\prime$ be two distinct limit points. Since $f$ is lower bounded
and the function value strictly decreases with the outer iteration $n$ by Claim 1,
we must have $f(y)=f(y^\prime)$.
Let $x^n$ and $x^{n+1}$ be two consecutive points generated by the outer iteration,
and suppose they are in sufficiently small neighborhoods of $y$ and $y^\prime$, respectively. 
Let $\alpha_n$ and $d^n$ be the step size and gradient given at \linref{lin:x^n_update}
of the main algorithm in the outer iteration $n$, i.e., $x^{n+1}=x^n-\alpha_nd^n$. 
We have $\alpha_n\norm{d^n}\ge\norm{y-y^\prime}-2\epsilon$
and hence $\alpha_n\norm{d^n}^2\ge(\norm{y-y^\prime}-2\epsilon)^2/\alpha_n$. Therefore, the ratio
\bdm
\rho_n = \frac{f(x^n)-f(x^n-\alpha_nd^n)}{\alpha_n\norm{d^n}^2}
\le\frac{\alpha_n[f(x^n)-f(x^n-\alpha_nd^n)]}{(\norm{y-y^\prime}-2\epsilon)^2}
\edm
can be sufficiently small (i.e., $\rho_n<\rho_0$) as because $\epsilon\to 0$, 
$f(x^n)\to f(y)$, $f(x^n-\alpha_nd^n)=f(x^{n+1})\to f(y^\prime)$, $f(y)=f(y^\prime)$ and $\norm{y-y^\prime}$
is a fixed value. This indicates that the condition at \linref{lin:rho^n_k>=rho_0} of \Call{LineSearch}{}
cannot be satisfied and hence $x^{n+1}$ cannot be in a sufficiently small neighborhood of $y^\prime$,
which leads to a contradiction. This proves the claim. 

Let $x^*$ be the unique limit point of $\{x^n\}^\infty_{n=1}$ by Claim 3. 
Let $\Thetadisc^n$ be the value of the dynamic set $\Thetadisc$ at the beginning of the outer iteration $n$, 
and $\Theta_0=\lim_{n\to\infty}\Thetadisc^n$ (note that $\Thetadisc^m\subseteq\Thetadisc^n$ for $m\le{n}$
and hence the limit exists).
Let $z^n_\theta$ be the (unique) representative point associated 
with branch $\theta$ for $\theta\in\Thetadisc^n$ at the beginning of the outer iteration $n$. 

Claim 4. For each $\theta\in\Theta_0$, either one of the following two cases should hold:
(1) $\lim_{n\to\infty} z^n_\theta = x^*$; (2) $\exists\epsilon>0$ such that
$\inf_{n\to\infty}\norm{z^n_\theta - x^*}\ge\epsilon$.

Proof of Claim 4. Note that $\Theta_0$ can be partitioned into two 
subsets $\Theta_1$ and $\Theta_2$. Branches in $\Theta_1$ (resp. $\Theta_2$) 
are visited by $\{x^n\}$ in finitely (resp. infinitely) many iterations.
By the rule of branch-point update in \Call{BranchPointUpdate}{},
$z^n_\theta$ is the closest point to $x^n$ from the set $\cD_\theta\cap\{x^i\}^n_{i=1}$.
Since $x^n\to{x^*}$, it follows that there exists $\epsilon>0$ such that $\inf_{n\to\infty}\norm{z^n_\theta - x^*}\ge\epsilon$
for $\theta\in\Theta_1$, and $\lim_{n\to\infty} z^n_\theta = x^*$ for $\theta\in\Theta_2$.
This proves the claim. 

Claim 5. Let $\Theta^*=\{\theta\in\Theta_0:\lim_{n\to\infty}z^n_\theta=x^*\}$, which
is well defined due to Claim 4.
Then $\Theta^*\subseteq\cM(x^*)$. Furthermore, for any outer iteration $n$ that is large enough,
there exists an inner iteration index $k=k_n$ such that $\cT^n_{k_n}=\Theta^*$. 

Proof of Claim 5. Condition 4 in \defref{def:piecewise-diff} implies $\Theta^*\subseteq\cM(x^*)$.
By the definition of $\Theta^*$, we conclude that $\Theta^*\subseteq\cT(x^n,r_0)$
for any $n$ that is large enough. Furthermore, there exists a constant $b_0>0$
such that $\max_{\theta\in\Theta^*}\norm{z^n_\theta-x^n}<b_0$ and  
$\max_{\theta\in\cT(x^n,r_0)\setminus\Theta^*}\norm{z^n_\theta-x^n}>b_0$
for large enough $n$. Therefore, by enumerating over elements in $\cZ^n$ based on
their distance from $x^n$, one can identify a number $k_n$ such that 
$\{z^n_\theta: \theta\in\Theta^*\}$ are the $k_n$ distinct nearest points from $x^n$ within the set $\cZ^n$,
and hence $\cT^n_{k_n}=\Theta^*$. This proves the claim.

Claim 6. There exists a constant $c_0>0$ and a large enough integer $N$ 
such that if $n>N$ and $\alpha_t\le{c_0}$ then $\cM(x^n-\alpha_td^n_k)\subseteq\Theta^*$,
where $n$ and $k$ are outer- and inner-iteration indices, and $\alpha_t$ is a step-size generated 
at \linref{lin:alpha-dec} in \Call{LineSearch}{}. 

Proof of Claim 6. We prove it by contradiction. Suppose the claim does not hold.
Then there exist subsequences 
$\{n_i\}$, $\{k_i\}$, $\{\alpha_{t_i}\}$ and $\{\theta_i\}$
such that $n_i\to\infty$, $\alpha_{t_i}\to0$
and $\theta_i\in\cM(x^{n_i}-\alpha_{t_i}d^{n_i}_{k_i})$ but $\theta_i\notin\Theta^*$.
Since the total number $|\Theta|$ of branch functions is finite. There exists a 
branch $\theta$ that appears infinitely many times in the subsequence $\{\theta_i\}$.
Due to the procedure \textproc{BranchPointUpdate} called at \linref{lin:update1} in \Call{LineSearch}{},
one has $\cM(x^{n_i}-\alpha_{t_i}d^{n_i}_{k_i})\subseteq\Theta_0$ by 
the definition of $\Theta_0$ and hence $\theta\in\Theta_0$.
Since $\lim_{i}x^{n_i}-\alpha_{t_i}d^{n_i}_{k_i}=x^*$, we conclude that $z^n_\theta\to{x^*}$,
which indicates that $\theta\in\Theta^*$ by the definition of $\Theta^*$ and it contradicts to the hypothesis $\theta\notin\Theta^*$.
This proves the claim. 

It is ready to prove the theorem by contradiction. Suppose the limit point $x^*$ is non-stationary.
Then we have $\norm{d^*}>0$, where $d^*=\qpvc{x^*}{\cM(x^*)}$. 
Let $k_n$ be the inner iteration index specified in Claim 5. 
For a large enough $n$, consider the outer- and inner-iteration $(n,k_n)$ for which $\cT^n_{k_n}=\Theta^*$. 
Consider the \Call{LineSearch}{} function called at the outer- and inner-iteration $(n,k_n)$
in the main algorithm \linref{lin:call-line-search}.
Note that $d^n_{k_n}$ is the optimal vector of \eqref{opt:qp-sum-lambda-grad-f}
with $\cT^n_{k_n}=\Theta^*$. By \propref{prop:d_limit}, 
we have $\norm{d^n_{k_n}}\to\norm{d^*}$ as $n\to\infty$
due to that $z^n_\theta\to{x^*}$ and the continuity of $\nabla{f}_\theta(x)$.
Note that by Claim 6, there exists a constant $c_0>0$ such that when $\alpha_t<{c_0}$, 
one has $\cM(x^n-\alpha_td^n_k)\subseteq\Theta^*$. Since $x^n\to{x^*}$,
one has $\cM(x^n)\subseteq\Theta^*\subseteq\cM(x^*)$ by Claim 5. 
Then the following inequalities hold for any $\theta\in\Theta^*$:
\beq
\ba
&\frac{f(x^n)-f_\theta(x^n-\alpha_t d^n_{k_n})}{\alpha_t\norm{d^n_{k_n}}^2} = \frac{f(x^n)-f(x^*)+f(x^*)-f_\theta(x^n-\alpha_t d^n_{k_n})}{\alpha_t\norm{d^n_{k_n}}^2} \\
&\ge\frac{f(x^n)-f(x^*)+f(x^*)-f_\theta(x^n)+\alpha_t\inner{\nabla f_\theta(x^n)}{d^n_{k_n}}-\frac{1}{2}\alpha^2_t\norm{d^n_{k_n}}^2}{\alpha_t\norm{d^n_{k_n}}^2} \\
&=\frac{f(x^n)-f(x^*)+f_\theta(x^*)-f_\theta(x^n)+\alpha_t\inner{\nabla f_\theta(x^n)}{d^n_{k_n}}-\frac{1}{2}\alpha^2_t\norm{d^n_{k_n}}^2}{\alpha_t\norm{d^n_{k_n}}^2} \\
&\ge\frac{f(x^n)-f(x^*)+f_\theta(x^*)-f_\theta(x^n)+\alpha_t\norm{d^n_{k_n}}^2-\frac{1}{2}\alpha^2_t\norm{d^n_{k_n}}^2}{\alpha_t\norm{d^n_{k_n}}^2}\quad \mathrm{\propref{prop:<u,vi>}}.
\ea
\eeq
The above inequality implies that if $\alpha_t\ge{c_1}$ for some constant $c_1>0$, one has
\beq
\ba
&\lim_n\frac{f(x^n)-f_\theta(x^n-\alpha_t d^n_{k_n})}{\alpha_t\norm{d^n_{k_n}}^2} \\
&\ge\lim_n\frac{f(x^n)-f(x^*)+f_\theta(x^*)-f_\theta(x^n)+\alpha_t\norm{d^n_{k_n}}^2-\frac{1}{2}\alpha^2_t\norm{d^n_{k_n}}^2}{\alpha_t\norm{d^n_{k_n}}^2} \\
&=\frac{\alpha_t\norm{d^*}^2-\frac{1}{2}\alpha^2_t\norm{d^*}^2}{\alpha_t\norm{d^*}^2} = 1-\frac{1}{2}\alpha_t.
\ea
\eeq 
Therefore, if $n$ is large enough, we have 
\beq\label{eqn:rhon>rho0}
\frac{f(x^n)-f_\theta(x^n-\alpha_t d^n_{k_n})}{\alpha_t\norm{d^n_{k_n}}^2}\ge\rho_0
\eeq
for any $\theta\in\Theta^*$ and any $\alpha_t$ that is small enough. 
Let $\theta^\prime\in\cM(x^n-\alpha_td^n_k)$. Then we have:
\beq\label{eqn:rhon>rho0_part2}
\frac{f(x^n)-f(x^n-\alpha_t d^n_{k_n})}{\alpha_t\norm{d^n_{k_n}}^2}=\frac{f(x^n)-f_{\theta^\prime}(x^n-\alpha_t d^n_{k_n})}{\alpha_t\norm{d^n_{k_n}}^2}
\overset{\eqref{eqn:rhon>rho0}}{\ge}\rho_0
\eeq 
for any small enough step size $\alpha_t$.
Since $\theta^\prime\in\Theta^*$, \eqref{eqn:rhon>rho0_part2} implies that 
once $\alpha_t$ decreases down below some constant $\beta>0$, 
the condition $\rho\ge\rho_0$ must hold and the while-loop should
terminate at the end of this iteration of $t$. So the termination step-size $\alpha_{t^\prime}$
must satisfy $\alpha_{t^\prime}\ge\mudec\beta$. 
Then the reduction of objective value is guaranteed 
to have a lower bound $\rho_0\mudec\beta\norm{d^n_{k_n}}^2\ge\rho_0\mudec\beta\norm{d^*}^2/2$.
This is a positive constant, which contradicts to $x^n\to{x^*}$. Therefore, $x^*$ is stationary.
\end{proof}

\section{An enhanced algorithm and analysis of the convergence rate}
\label{sec:CID-JGD}
In this section, we develop and analyze an enhanced branch-information-driven
gradient descent algorithm (EBIGD). See \algoref{alg:global_and_local_converg} for details.
This algorithm can achieve local linear convergence under certain conditions.
To analyze the rate of convergence, we consider ecodable piecewise-smooth functions that are locally-max representable
defined as follows: 
\begin{definition}[locally-max representability]\label{def:local-max}
An encodable function $(f,\cM)$ is \textit{locally-max-representable} in an open neighborhood $A$
of a local minimizer $x^*$ if for any $x\in{A}$, the objective $f$ has the representation $f(x)=\max\{f_\theta(x): \theta\in\tM(x)\}$.    
\end{definition}
A notion that is similar to \defref{def:local-max} has been used in \cite{lewis2023-survey-descent} (Definition~3) to analyze the convergence rate
of a survey descent method. We make the following two additional assumptions throughout the section:
\begin{assumption}\label{ass:twice-diff}
Every branch function is twice differentiable in its domain of definition, i.e., the Hessian matrix 
$\nabla^2{f}_\theta(x)$ exists for any $x\in\tD_\theta$ and $\theta\in\Theta$.
\end{assumption}
\begin{assumption}\label{ass:verify-theta-in-tM}
For any $x\in\mR^n$, it is possible for an algorithm to verify if a given branch $\theta$ is in $\tM(x)$.
\end{assumption}
The above assumption is used in \algoref{alg:global_and_local_converg}.
The motivation of it is to verify if the current point $x$ is in the definition domain
$\tD_\theta$ of a branch $\theta$ that is not active at $x$. 
If $x\notin\tD_\theta$, the function value $f_\theta(x)$ and gradient $\nabla{f}_\theta(x)$
are not well defined, and hence cannot be used in the algorithm. However, in a small enough
neighborhood $A$ of $x^*$, all the active branches in this region belong to $\tM(x)$ for any $x\in{A}$. 
   
It can be shown that in a small neighborhood of a local minimizer $x^*$,
the distance from a point $x$ to $x^*$ is bounded by $\norm{\vc{\QPt{x}{\cM(x^*)}}}+\gap(x,\cM(x^*))$ 
up to a constant factor, where $\gap(x,\cC)=\max_{\theta\in\cC}f_\theta(x)-\min_{\theta\in\cC}f_\theta(x)$
for any set $\cC\subseteq\tM(x)$ of branch functions. Therefore, it amounts to reduce the norm 
of the joint gradient and the gap. However, the set $\cM(x^*)$ is not fully known
at a point $x\neq{x^*}$, which is a challenge to the algorithm design.

Based on the above motivation, the EBIGD algorithm works as follows:  At each iteration, the algorithm collects 
a set $\cC$ of branch functions sequentially through the function \textproc{BranchSelection}. 
It first initializes $\cC$ to just include an active 
branch at the current point $x$, and computes a corresponding joint gradient $d=\vc{\QPt{x}{\cC}}$. 
Using a sufficiently small constant step size $\alpha$, it then checks if the function value $f(x-\alpha{d})$ 
has a sufficient decrease compared to $f(x)$. If sufficient decrease is not
met while $x-\alpha{d}$ is in an active branch $\theta^\prime\notin\cC$, the algorithm 
adds $\theta^\prime$ to $\cC$ and computes a new joint gradient. 
This procedure is repeated until the sufficient decrease is satisfied. 
The function \textproc{BranchSelection} returns an eventual set $\cC$ of branches
encountered in this process and the corresponding joint gradient $d$. 
It can be shown that the set $\cC$ contains a complete basis (\defref{def:opt-joint-domain,opt-set}) 
of $\cM(x^*)$ if $x$ is in a sufficiently small neighborhood of $x^*$.
Using this information, the algorithm then executes two functions $\textproc{JointGradientReduction}$
and $\textproc{GapReduction}$ separately to compare which one leads to larger function value reduction. 
The $\textproc{JointGradientReduction}$ returns a candidate point $x-\alpha{d}$, and it can be proved that
the norm $\norm{\vc{\QPt{x}{\cC}}}$ of the joint gradient will decrease if the solution $x$ moves to this candidate point. 
The function $\textproc{GapReduction}$ is designed to reduce the gap $\gap(x,\cC)$ by solving a 
convex program involving a simple quartic term to determine a descent direction. 

A local minimizer $x^*$ is Clarke stationary. In general, it is possible to find a proper subset $\cC$ of $\cM(x^*)$
such that $\bs{0}\in\conv(\{\nabla{f}_\theta(x^*):\;\theta\in\cC\})$. In this case, $\cM(x^*)$ is more than necessary
to characterize the Clarke stationarity of $x^*$. This motivates the definition of a complete basis as follows:
\begin{definition}\label{def:opt-joint-domain,opt-set}
For an encodable function $(f,\cM)$ and a local minimizer $x^*$, 
A subset $\cC\subset\cM(x^*)$ is a \textit{complete basis} of branch functions if 
it satisfies the following two conditions:
\begin{itemize}
	\item[(a)] $\norm{\vc{\QPt{x^*}{\cC}}}=0$; 
	\item[(b)] either $|\cC|=1$ or $\norm{\vc{\QPt{x^*}{\cC^\prime}}}>0$ for any $\cC^\prime\subsetneq\cC$.
\end{itemize}
\end{definition}
In words, a complete basis is a minimal subset $\cC$ of $\cM(x^*)$ satisfying $\bs{0}\in\conv(\{\nabla{f}_\theta(x^*):\;\theta\in\cC\})$.
\begin{observation}\label{obs:complete-basis}
If $x^*$ is a local minimizer of $f$, there exists at least one complete basis of branch functions at $x^*$.
Furthermore, in a sufficiently small neighborhood $A$ of $x^*$, there exists a constant $c>0$ such that
$\norm{\vc{\QPt{x}{\cC^\prime}}}>c$ for any $x\in{A}$ and any $\cC^\prime\subseteq\cM(x^*)$ that does not contain a complete basis of $\cM(x^*)$.
\end{observation}
The following observation is useful in gap reduction analysis.
\begin{observation}\label{obs:bi-polar}
Given $\{\nabla f_i(x^*): i\in\cM(x^*)\}$ for a local minimizer $x^*$, there exists a constant $c>0$
such that for any complete basis $\cC\subseteq\cM(x^*)$ at $x^*$ with $|\cC|\ge2$,
and any vector $v\in\sp\{\nabla f_i(x^*): i\in\cC\}$ with $\norm{v}=1$, 
there exist $i,j\in\cC$ such that $\inner{\nabla f_i(x^*)}{v}\ge{c}$
and $\inner{\nabla f_j(x^*)}{v}\le{-c}$.
\end{observation}
Suppose $\cC\subseteq\cM(x^*)$ is a subset that contains a complete basis. Consider
the subspace $S=\sp\{\nabla{f}_\theta(x^*): \theta\in\cC\}$. If the dimension of $S$ 
is smaller than $n$, we need to impose a regularity condition
on the behavior of function $f_\theta(x^*+\alpha{v})$ for $v\in{S}^\perp$ and $0<\alpha<\ll1$,
which is needed for linear convergence. This regularity condition is defined as follows:   
\begin{definition}[complementary positive-definite condition]\label{def:comp-pos-def}
The complementary positive-definite condition is satisfied at a local minimizer $x^*$ of an encodable function
$(f,\cM)$ if for any subset $\cC\subseteq\cM(x^*)$ that contains a complete basis at $x^*$,
the inequality $v^\top\nabla^2 f_\theta(x^*)v>0$ holds for every $\theta\in\cC$ 
and all $v\in{V}^\perp\setminus\{\bs{0}\}$, where $V=\sp\{\nabla{f}_\theta(x^*):\theta\in\cC\}$.
\end{definition}
The necessity of the complementary positive-definite condition can be seen from an extremal case. 
Suppose $\cM(x^*)=\{\theta\}$. It follows that $f$ behaves like a smooth branch function $f_\theta$
in a small neighborhood of $x^*$. One can verify that the complementary positive-definite condition
reduces to $\nabla^2f_\theta(x^*)\succ{0}$ in this case, and the later one is a well known sufficient condition
that ensures an algorithm can be designed to achieve local linear convergence.
\begin{definition}[dimension consistency condition]\label{def:sym}
The dimension consistency condition is satisfied at a local minimizer $x^*$ of an encodable function
$(f,\cM)$ if there exists a complete basis $\cC$ at $x^*$ such that 
$\sp\{\nabla f_\theta(x^*): \theta\in\cC\}=\sp\{\nabla f_\theta(x^*): \theta\in\cM(x^*)\}$.
\end{definition}
The following proposition provides a complexity analysis of the function $\textproc{BranchSelection}$.
\begin{proposition}\label{prop:comp-select-while-loop}
Let $(f,\cM)$ be an encodable locally-max-representable function and $x^*$ be a local minimizer
of $f$. Let $x$ be a point in a small neighborhood of $x^*$.
For a sufficient small step size $\alpha$, 
the while-loop in the function \textproc{BranchSelection}\emph{($x,\alpha,\rho_0$)}
of \algoref{alg:global_and_local_converg} terminates in a bounded number of iterations,
and the set $\cC$ of branches returned by this function satisfies $\cC\subseteq\cM(x^*)$. 
\end{proposition}
\begin{proof}
Since the definition domain $\tD_\theta$ is an open set for each branch function $f_\theta$. 
By selecting a $r>0$ such that $\mB(x^*,r)\subseteq\cap_{\theta\in\tM(x^*)}\tD_\theta$,
one can guarantee that $\tM(x)=\tM(x^*)$ for every $x\in\mB(x^*,r)$.
Since $x$ is in a small neighborhood of $x^*$, $\alpha$ is sufficiently small
and $\norm{d_i}\le{G_f}$, by the condition 4 of \defref{def:piecewise-diff}, one has 
$\cM(x)\subseteq\cM(x^*)$ and $\cC_i\subseteq\cM(x^*)$ for every iteration $i$ 
in the function \textproc{BranchSelection}.
It implies that after at most $|\cM(x^*)|$ iterations
in the while-loop of \textproc{BranchSelection}, the following
condition should be satisfied: $\theta_{i+1}\subseteq\cC_i$ with $i\le|\cM(x^*)|$. 

Since $\theta_i,\theta_{i+1}\in\cC_i$ and $d_i=\qpvc{x}{\cC_i}$,
one has $\inner{\nabla f_{\theta_i}(x)}{d_i}\ge\norm{d_i}^2$ and 
$\inner{\nabla f_{\theta_{i+1}}(x)}{d_i}\ge\norm{d_i}^2$ by \propref{prop:<u,vi>}. 
Let $\theta_1$ be the active branch at $x$.
If $\norm{d_i}=0$, the loop ends, otherwise, one has
\bdm
\ba
&\rho=\frac{f(x)-f(x-\alpha{d_i})}{\alpha\norm{d_i}^2}=\frac{f_{\theta_1}(x)-f_{\theta_{i+1}}(x-\alpha{d_i})}{\alpha\norm{d_i}^2} \\
&\ge\frac{f_{\theta_{i+1}}(x)-f_{\theta_{i+1}}(x-\alpha{d_i})}{\alpha\norm{d_i}^2}\ge1-\frac{1}{2}{\alpha}L_\Theta,
\ea
\edm
where we use the property of a locally-max-representable function to get the first inequality.
It indicates that $\rho\ge\rho_0$ if $\alpha$ is small enough, i.e., $\alpha<2(1-\rho_0)/L_\Theta$. 
This shows that the while-loop should terminate within $|\cM(x^*)|$ iterations.
\end{proof}
Suppose $x$ is in a small neighborhood of $x^*$ and let $d=x-x^*$. Let $\cC$ be a complete basis of $\cM(x^*)$
returned by $\textproc{BranchSelection}$, and let $V=\sp\{\nabla{f}_i(x^*): i\in\cC\}$. 
Note that $d=\proj_Vd+\proj_{V^\perp}d$. The \textproc{GapReduction} is designed to reduce the contribution
of $\proj_Vd$ in the objective value in the case that $\norm{\proj_Vd}=O(\norm{d})$. 
The quartic convex program \eqref{opt:dir-recovery} is used to approximately
recover $\proj_Vd$. Specifically, if $(z^*,s^*)$ is an optimal solution of \eqref{opt:dir-recovery},
the linear expression $z^*+\inner{\nabla f_i(x)}{s^*}$ approximates the first-order Taylor expansion of $f_i(x)$ near $x^*$,
where $z^*\approx{f^*}$ and $s^*\approx\proj_Vd$. 
Since $\inner{\nabla f_i(x)}{s^*}\approx\inner{\nabla f_i(x)}{s^*+u}$ for a $u\in{V}^\perp$,
a quartic term $\norm{s}^4$ is added to the convex program to penalize component in $s^*$ 
that lies in ${V}^\perp$.
Intuitively, moving along $-s^*$ should reduce the objective value contributed by $\proj_Vd$.
This result is formally shown in the following lemma: 
\begin{lemma}\label{lem:gap-red}
Let $(f,\cM)$ be an encodable locally-max-representable function and $x^*$ be a local minimizer
of $f$. Let $x$ be a point in a neighborhood of $x^*$. Let $d=x-x^*$ and $(z^*,s^*)$ be an optimal solution of
\beq\label{opt:dir-recovery}
\min_{z\in\mR,s\in\mR^n}\; \sum_{i\in\cC}(z+\inner{\nabla f_i(x)}{s} - f_i(x))^2 + \norm{s}^4,
\eeq
where $\cC$ is a subset of $\cM(x^*)$ that contains a complete basis. 
Then the optimal solution satisfies $\norm{s^*}\le{O}(\norm{d})$ and $|z^*-f^*|\le{O}(\norm{d}^2)$ as $x\to{x^*}$. 
Furthermore, the following inequality holds for any $0<\gamma<1$ if $\cM(x-\gamma{s^*})\cap\cC$ is nonempty:
\beq\label{eqn:f(x)-f(x-as)}
f(x) - f(x-\gamma{s^*})\ge\gamma\inner{\nabla f_\theta(x^*)}{d} + O(\norm{d}^2)\quad\mathrm{as}\;\;x\to{x^*},
\eeq
where $\theta$ is an active branch at $x$. 
\end{lemma}
\begin{proof}
Let $d=x-x^*$. For any $i\in\cM(x^*)$, the gradient $\nabla{f}_i(x)$ and the function value $f_i(x)$ can be represented
using Taylor expansion:
\beq\label{eqn:taylor-expan}
\nabla{f_i}(x) = g_i + H_id + \tau_i,\quad
f_i(x) = f^* + \inner{g_i}{d} + c_i,
\eeq
where $f^*=f(x^*)$, $g_i=\nabla{f_i}(x^*)$, $H_i=\nabla^2f_i(x^*)$, 
$\tau_i\in\mR^n$ satisfying $\norm{\tau_i}\le{K}\norm{d}^2$
for some constant $K>0$, and $c_i\in\mR$ satisfying $|c_i|\le{L_\Theta}\norm{d}^2/2$.
Denote $\opt$ the optimal value of \eqref{opt:dir-recovery}, and denote $F(z,s)$ the objective
function of \eqref{opt:dir-recovery}. We first note that
\bdm
\ba
\opt&=F(z^*,s^*)\le{F}(f^*,d) \\
&\le\sum_{i\in\cC}\big(f^*+\inner{g_i+H_id+\tau_i}{d} - f^*-\inner{g_i}{d}-c_i\big)^2 + \norm{d}^4 \\
&=\sum_{i\in\cC}\big(d^\top{H}_id + \inner{\tau_i}{d}-c_i\big)^2 + \norm{d}^4 \\
&\le\sum_{i\in\cC}\left(\norm{H_i}\norm{d}^2+K\norm{d}^3+\frac{1}{2}L_\Theta\norm{d}^2 \right)^2 + \norm{d}^4 \\
&\le{M_1}\norm{d}^4,
\ea
\edm
for some constant $M_1>0$. The expression of $F(z^*,s^*)$ is written as:
\bdm
F(z^*,s^*)=\sum_{i\in\cC}\big(z^*+\inner{g_i+H_id+\tau_i}{s^*} - f^*-\inner{g_i}{d}-c_i\big)^2 + \norm{s^*}^4.
\edm
First, the inequality $F(z^*,s^*)\le{M}_1\norm{d}^4$ implies $\norm{s^*}\le{M}^{1/4}_1\norm{d}<1$ for small $\norm{d}$.
This proves the statement $\norm{s^*}\le{O}(\norm{d})$ in the lemma.
Next, we derive a lower bound for $F(z^*,s^*)$ as:
\beq\label{eqn:F(z,s)_ineq_1}
\ba
F(z^*,s^*)&\ge\sum_{i\in\cC}\big(z^*+\inner{g_i+H_id+\tau_i}{s^*} - f^*-\inner{g_i}{d}-c_i\big)^2  \\
&=\sum_{i\in\cC}\big(z^*-f^*+\inner{g_i}{s^*}-\inner{g_i}{d} + d^{\top}H_is^* + \inner{\tau_i}{s^*} - c_i\big)^2  \\
&\ge\sum_{i\in\cC}\frac{1}{2}\big(z^*-f^*+\inner{g_i}{s^*}-\inner{g_i}{d} + d^{\top}H_is^*\big)^2 - \big(\inner{\tau_i}{s^*} - c_i\big)^2  \\
&\ge\sum_{i\in\cC}\frac{1}{2}\big(z^*-f^*+\inner{g_i}{s^*}-\inner{g_i}{d} + d^{\top}H_is^*\big)^2 - \left(K+\frac{1}{2}L_\Theta\right)^2|\cC|\cdot\norm{d}^4, 
\ea
\eeq
where we have used the algebraic inequality $(a+b)^2\ge{a^2/2-b^2}$ to derive the inequality in the third line,
and used $\norm{s^*}<1$, and the bounds of $\norm{\tau_i}$ and $|c_i|$ to get the last inequality. 
Let $V=\sp\{g_i:i\in\cC\}$, and consider the representation  
$z^*=f^*+t$ for some $t\in\mR$, $d=d_1+d_2$ with $d_1\in{V}$ and $d_2\in{V}^\perp$,
and $s^*=d_1+q$ for some $q\in\mR^n$.  
Substitute these representations into \eqref{eqn:F(z,s)_ineq_1} leads to
\bdm
\ba
F(z^*,s^*)&\ge\frac{1}{2}\sum_{i\in\cC}(t+\inner{g_i}{q} +  d^{\top}H_is^*)^2 - \left(K+\frac{1}{2}L_\Theta\right)^2|\cC|\cdot\norm{d}^4 \\
&\ge\frac{1}{2}\sum_{i\in\cC}\left[\frac{1}{2}(t+\inner{g_i}{q})^2 -  (d^{\top}H_is^*)^2\right] - \left(K+\frac{1}{2} L_\Theta\right)^2|\cC|\cdot\norm{d}^4 \\
&\ge\frac{1}{4}\sum_{i\in\cC}(t+\inner{g_i}{q})^2 - \frac{1}{2}\sum_{i\in\cC}\sqrt{M_1}\cdot\norm{H_i}^2\norm{d}^4 - \left(K+\frac{1}{2} L_\Theta\right)^2|\cC|\cdot\norm{d}^4.
\ea
\edm
Combining the above inequality with $F(z^*,s^*)\le{M}_1\norm{d}^4$ implies 
\beq\label{eqn:gamma+inner_prod}
(t+\inner{g_i}{q})^2\le{M_2}\norm{d}^4\quad\forall i\in\cC,
\eeq
for some constant $M_2>0$. 

We claim that $|\inner{g_i}{q}|\le2\sqrt{M_2}\norm{d}^2$
for all $i\in\cC$. We prove the claim by contradiction. Suppose there exists a $i\in\cC$
such that $|\inner{g_i}{q}|>2\sqrt{M_2}\norm{d}^2$. 
Without loss of generality, assume $\inner{g_i}{q}>2\sqrt{M_2}\norm{d}^2$.
The \eqref{eqn:gamma+inner_prod}
for the branch $i$ implies that $t<-\sqrt{M_2}\norm{d}^2$.
On the other hand, since $\cC$ contains a complete basis, 
by \obsref{obs:bi-polar}, there exists a branch $j\in\cC$
such that $\inner{g_j}{q}\le0$. Then $\left(t+\inner{g_j}{q}\right)^2>M_2\norm{d}^4$,
which violates \eqref{eqn:gamma+inner_prod} for $j$. Therefore, the claim holds. 
Combining \eqref{eqn:gamma+inner_prod} and the claim, we obtain
$|t|\le3\sqrt{M_2}\norm{d}^2$. This proves the statement $|z^*-f^*|\le{O}(\norm{d}^2)$ in the lemma. 

We also claim that $\norm{q}\le{O}(\norm{d}^2)$. To prove it, 
we first notice that $q\in{V}$. Otherwise, if $q=q_1+q_2$ for some $q_1\in{V}$ and $q_2\in{V^\perp}$. 
The solution $(z^*,d_1+q_1)$ gives a better objective than $(z^*,s^*)$ for the optimization problem
\eqref{opt:dir-recovery}, which is a contradiction. Select an arbitrary $i\in\cC$. 
By \obsref{obs:bi-polar}, there exists a constant $c>0$ satisfying $|\inner{g_i}{q}|>c\norm{q}$.
Using the result from the previous claim, we conclude that
\bdm
c\norm{q}<|\inner{g_i}{q}|\le2\sqrt{M_2}\norm{d}^2,
\edm
which proves $\norm{q}\le{O}(\norm{d}^2)$.

To prove \eqref{eqn:f(x)-f(x-as)}, suppose $i$ and $j$ are active branch(es) at $x$
and $x-\gamma{s^*}$, respectively, where $i$ and $j$ can be identical. We also assume $j\in\cC$ based on
the hypothesis made for \eqref{eqn:f(x)-f(x-as)}. This means $g_i,g_j\in{V}$.
Since $j\in\cC$ and $\cC\subseteq\cM(x^*)$, we have $j\in\cM(x^*)$.
Since $f_i(x)\ge{f_j}(x)$, it follows that $\inner{g_i-g_j}{d_1}\ge{O}(\norm{d}^2)$ using Taylor expansion. 
Using the representation $s^*=d_1+q$ with $\norm{q}\le{O}(\norm{d}^2)$, we obtain
\beq\label{eqn:f_i(x-gamma*s)}
\ba
f_\theta(x-\gamma{s^*})&=f_\theta(x)-\gamma\inner{\nabla f_\theta(x)}{s^*} + O(\norm{d}^2) \\
&\overset{\eqref{eqn:taylor-expan}}{=}f^*+\inner{g_\theta}{d} -\gamma\inner{g_\theta+H_\theta d}{s^*} + O(\norm{d}^2) \\
&=f^*+\inner{g_\theta}{d_1}-\gamma\inner{g_\theta}{d_1+q} + O(\norm{d}^2) \quad \textrm{ using } \norm{s^*}\le{O}(\norm{d})  \\
&=f^*+(1-\gamma)\inner{g_\theta}{d_1} + O(\norm{d}^2)\quad\textrm{ using } \norm{q}\le{O}(\norm{d}^2),
\ea
\eeq
for any $\theta\in\cM(x^*)$, and in particular for $\theta\in\{i,j\}$.
Since $j$ is an active branch at $x-\gamma{s^*}$,
one has $f_i(x-\gamma{s^*})\le{f_j}(x-\gamma{s^*})$.
Substituting \eqref{eqn:f_i(x-gamma*s)} into this inequality gives
$\inner{g_i-g_j}{d_1}\le{O}(\norm{d}^2)$. 
Combining with $\inner{g_i-g_j}{d_1}\ge{O}(\norm{d}^2)$ established earlier, 
we have proved
\beq\label{eqn:|<gi-gj,d>|<=d^2}
|\inner{g_i-g_j}{d_1}|\le{O}(\norm{d}^2).
\eeq
We can derive the following bound
\bdm
\ba
&f(x) - f(x-\gamma{s^*})= f_i(x) - f_j(x-\gamma{s^*}) \\
&=\big(f^*+\inner{g_i}{d}\big) - \big(f^* + \inner{g_j}{d-\gamma s^*} \big) + O(\norm{d}^2) \quad \textrm{ using } \norm{s^*}\le{O}(\norm{d}) \\
&=\big(f^*+\inner{g_i}{d_1}\big) - \big(f^*+ \inner{g_j}{d_1} - \gamma\inner{g_j}{d_1+q} \big) + O(\norm{d}^2)\quad \textrm{ using } d_1=\proj_V d \\
&=\inner{g_i}{d_1} - (1-\gamma)\inner{g_j}{d_1} + O(\norm{d}^2) \quad \textrm{ using } \norm{q}\le{O}(\norm{d}^2) \\
&=\inner{g_i-g_j}{d_1} + \gamma\inner{g_j}{d_1} + O(\norm{d}^2) \\
&\overset{\eqref{eqn:|<gi-gj,d>|<=d^2}}{=}\gamma\inner{g_i}{d_1} + O(\norm{d}^2) \\
&=\gamma\inner{g_i}{d} + O(\norm{d}^2),
\ea
\edm
which concludes the proof.
\end{proof}
In the case that $\proj_{V^\perp}d$ dominates $\proj_{V}d$, the \textproc{JointGradientReduction}
should be more effective in reducing the objective value. This result is given in the following lemma:  
\begin{lemma}\label{lem:vec>sigma*d}
Let $(f,\cM)$ be an encodable locally-max-representable function and $x^*$ be a local minimizer
of $f$. Let $x$ be a point in a small neighborhood of $x^*$ and $d=x-x^*$. Let $\cC$ 
be a subset of $\cM(x^*)$ containing a complete basis. Suppose $d=d_1+d_2$, 
where $d_1\in{V}$, $d_2\in{V^\perp}$, $\norm{d_2}>0$ and $V=\sp\{\nabla{f}_i(x^*): i\in\cC\}$. 
Suppose the complementary positive-definite condition holds at $x^*$. If $\norm{d_1}\le{O}(\norm{d}^2)$,
then there exists a constant $\sigma>0$ such that
\beq
\norm{\sum_{i\in\cC}\lambda_i\nabla f_i(x)}\ge\sigma\norm{d},
\eeq
for any $\lambda_i\ge0$, $\sum_{i\in\cC}\lambda_i=1$.
\end{lemma}
\begin{proof}
Let $g_i=\nabla{f}_i(x^*)$, $H_i=\nabla^2f_i(x^*)$ for $i\in\cC$, 
and $V=\sp\{g_i: i\in\cC\}$. For any coefficient vector $\lambda$ satisfying
$\sum_{i\in\cC}\lambda_i=1$ and $\lambda_i\ge0$, using Taylor expansion leads to
\beq
\ba
&\norm{\sum_{i\in\cC}\lambda_i\nabla f_i(x)}=\norm{\sum_{i\in\cC}\lambda_i(g_i+H_id)} + O(\norm{d}^2) \\
&\ge\norm{\proj_{V^\perp}\sum_{i\in\cC}\lambda_i(g_i+H_id)} + O(\norm{d}^2) \\
&=\norm{\proj_{V^\perp}\sum_{i\in\cC}\lambda_iH_id} + O(\norm{d}^2) \quad\textrm{using }\sum_{i\in\cC}\lambda_ig_i\in{V} \\
&\ge\norm{\proj_{V^\perp}\sum_{i\in\cC}\lambda_iH_id_2} - \norm{\proj_{V^\perp}\sum_{i\in\cC}\lambda_iH_id_1} + O(\norm{d}^2) \\
&\ge\norm{\proj_{V^\perp}\sum_{i\in\cC}\lambda_iH_id_2} + O(\norm{d}^2) \quad\textrm{using }\norm{d_1}\le{O}(\norm{d}^2).
\ea
\eeq
Let $H=\sum_{i\in\cC}\lambda_iH_i$, There exists an orthogonal matrix $Q$ such that 
\bdm
QHQ^{-1}=
\begin{bmatrix}
H_{11} & H_{12} \\
H_{21} & H_{22}
\end{bmatrix},
\edm
and for any $v\in\mR^n$
\beq\label{eqn:Q-v-proj}
Qv = 
\begin{bmatrix}
v_1 \\
v_2
\end{bmatrix},
\quad
Q\cdot\proj_Vv=
\begin{bmatrix}
v_1 \\
0
\end{bmatrix}
,
\quad 
Q\cdot\proj_{V^\perp}v=
\begin{bmatrix}
0 \\
v_2
\end{bmatrix}.
\eeq
In particular,
\beq\label{eqn:preserve_norm}
Qd_2=
\begin{bmatrix}
0 \\
d^\prime_2
\end{bmatrix},
\quad 
\norm{d^\prime_2}=\norm{d_2}.
\eeq
Furthermore, the three equations in \eqref{eqn:Q-v-proj} also implies that
\beq\label{eqn:Q^-1_proj}
Q^{-1}\begin{bmatrix}
0 \\
v_2
\end{bmatrix}
=\proj_{V^\perp}v=\proj_{V^\perp}Q^{-1}\begin{bmatrix}
v_1 \\
v_2
\end{bmatrix}
\eeq
Therefore, we have
\begin{displaymath}
\begin{aligned}
&\proj_{V^\perp}\sum_{i\in\cC}\lambda_iH_id_2=\proj_{V^\perp}Hd_2 \\
&=\proj_{V^\perp}Q^{-1}\begin{bmatrix}
H_{11} & H_{12} \\
H_{21} & H_{22}
\end{bmatrix}
\begin{bmatrix}
0 \\
d^\prime_2
\end{bmatrix} 
=\proj_{V^\perp}Q^{-1}
\begin{bmatrix}
H_{12}d^\prime_2 \\
H_{22}d^\prime_2
\end{bmatrix} \\
&\overset{\eqref{eqn:Q^-1_proj}}{=}Q^{-1}
\begin{bmatrix}
0 \\
H_{22}d^\prime_2
\end{bmatrix}.
\end{aligned}
\end{displaymath}
It then follows that
\bdm
\norm{\proj_{V^\perp}\sum_{i\in\cC}\lambda_iH_id_2}=\norm{Q^{-1}
\begin{bmatrix}
0 \\
H_{22}d^\prime_2
\end{bmatrix}}
\overset{\eqref{eqn:preserve_norm}}{=}\norm{H_{22}d^\prime_2},
\edm
where we use the fact that $Q$ is an orthogonal matrix. 
The complementary positive-definite condition implies that $H_{22}\succ0$,
and hence $\norm{H_{22}d^\prime_2}\ge2\sigma\norm{d^\prime_2}=2\sigma\norm{d_2}\ge\sigma\norm{d}$ 
for some constant $\sigma>0$, where we use the hypothesis that $d_1\le{O}(\norm{d}^2)$ in the last inequality.
This concludes the proof.
\end{proof}
\begin{theorem}[local linear convergence]\label{thm:linear_rate}
Let $(f,\cM)$ be an encodable locally-max-representable function and $x^*$ be a local minimizer
of $f$.
Suppose the complementary positive-definite condition and the dimension consistency condition
hold at $x^*$.
If the input step size $\alpha$ is sufficiently small, 
the \algoref{alg:global_and_local_converg} achieves a linear rate of convergence
in a small neighborhood of $x^*$.
\end{theorem}
\begin{proof}
Let $g_i=\nabla{f_i}(x^*)$ for $i\in\cM(x^*)$. We need to prove two claims. 
Claim 1. If $d=x-x^*$ is sufficiently small, then there exists
a constant $M$ such that for any subset $\cC\subseteq\cM(x^*)$ containing 
a complete basis at $x^*$, and for any coefficient vector $\lambda$ with $\sum_{i\in\cC}\lambda_i=1$ 
and $\lambda_i\ge0$ for all $i\in\cC$, the following inequality holds
\beq
\norm{\qpvc{x}{\cC}}\le{M}\norm{d}.
\eeq 
Proof of Claim 1. Suppose $\cB\subseteq\cC$ is a complete basis
and $\lambda^*\in\mR^{|\cB|}$ is a coefficient vector satisfying
$\lambda^*_i>0$, $\sum_{i\in\cB}\lambda^*_i=1$ and
$\sum_{i\in\cB}\lambda^*_ig_i=0$.
Define a coefficient vector $\lambda\in\mR^{|\cC|}$ such that 
$\lambda_i=\lambda^*_i$ for all $i\in\cB$ and $\lambda_i=0$
for all $i\in\cC\setminus\cB$. It follows that
\bdm
\ba
\norm{\qpvc{x}{\cC}}&\le\norm{\sum_{i\in\cC}\lambda_i\nabla f_i(x)}=\norm{\sum_{i\in\cB}\lambda^*_i\nabla f_i(x)} \\
&=\norm{\sum_{i\in\cB}\lambda^*_i(g_i+H_id)} + O(\norm{d}^2) \\
&=\norm{\sum_{i\in\cB}\lambda^*_iH_id} + O(\norm{d}^2) \\
&\le{M}\norm{d},
\ea
\edm
for some constant $M>0$, concluding the proof of the claim.

Claim 2. Suppose $d=x-x^*$ is sufficiently small and $\cC\subseteq\cM(x^*)$ contain a complete basis.  
If $\theta$ is an active branch at $x$ satisfying $\theta\in\cC$ and $\inner{g_\theta}{d}\le{O}(\norm{d}^2)$,
then $|\inner{g_i}{d}|\le{O}(\norm{d}^2)$ for every $i\in\cM(x^*)$. \newline
Proof of Claim 2. Suppose $\inner{g_\theta}{d}\le{a_1}\norm{d}^2$ for some constant $a_1>0$.
Let $\cB$ be any complete basis (not necessarily contained in $\cC$). 
We first need to show that $\norm{\proj_Vd}\le{O}(\norm{d}^2)$,
where $V=\sp\{g_i: i\in\cB\}$.
We prove it by contradiction. Suppose $\norm{\proj_Vd}\ge{a_2}\norm{d}^2$ for some 
constant $a_2>0$ that is large enough but independent of $\norm{d}$. 
Since $\cB$ is a complete basis, by \obsref{obs:bi-polar} there exists
a branch $j\in\cB$ such that $\inner{g_j}{\proj_Vd}\ge{a_3}\norm{\proj_Vd}$
for some constant $a_3>0$ that is independent of $\cB$. It then follows that
\beq\label{eqn:f_j-f_theta}
\ba
&f_j(x^*+d) - f_\theta(x^*+d) \ge \inner{g_j}{d} - \inner{g_\theta}{d} - L_\Theta\norm{d}^2 \\
&=\inner{g_j}{\proj_Vd} - \inner{g_\theta}{d} - L_\Theta\norm{d}^2\ge(a_2a_3-a_1-L_\Theta)\norm{d}^2. 
\ea
\eeq
Since $a_2$ is large enough by the hypothesis, one has $a_2a_3-a_1-L_\Theta>0$. It implies that 
$\theta$ cannot be an active branch at $x$ following the locally-max-representability of $f$ at $x^*$. 
Therefore, we have shown $\norm{\proj_Vd}\le{O}(\norm{d}^2)$ and hence
\bdm
|\inner{g_i}{d}|=|\inner{g_i}{\proj_Vd}| \le{O}(\norm{d}^2)
\edm
for every $i\in\cB$. 
We are ready to prove the claim.
Let $V_0=\sp\{g_i: i\in\cM(x^*)\}$. 
Due to the dimension consistency condition, there exists a basis $\cB^\prime$ at $x^*$ such that
$V_0=\sp\{g_i: i\in\cB^\prime\}$.
Since $\norm{\proj_{V_0}d}\le{O}(\norm{d}^2)$ proved earlier, one has $\norm{\proj_{V^\perp_0}d}\ge\norm{d}/2$.
Again, by \obsref{obs:bi-polar} there exist a $\theta_3\in\cB^\prime$ and 
a constant $a_3>0$ such that $\inner{g_{\theta_3}}{\proj_{V_0}d}\ge{a_3}\norm{\proj_{V_0}d}$ 
for some constant $a_3>0$. Note that
\bdm
f_\theta(x^*+d) - f_{\theta_3}(x^*+d)\ge0,
\edm
which is due to $\theta\in\cM(x)$ and $\theta_3\in\cB^\prime\subseteq\cM(x^*)\subseteq\tM(x)$ as $x$ is very close to $x^*$.
Using Taylor expansion, this inequality implies that
\bdm
\inner{g_\theta}{d} - \inner{g_{\theta_3}}{d}+\frac{1}{2}L_\Theta\norm{d}^2\ge0.
\edm
Since $\inner{g_\theta}{d}\le{O}(\norm{d}^2)$ by the hypothesis of Claim 2, the above inequality
implies $\inner{g_{\theta_3}}{d}\le{O}(\norm{d}^2)$.
Since $\inner{g_{\theta_3}}{d}$ can be decomposed as
\bdm
\ba
\inner{g_{\theta_3}}{d} &= \inner{g_{\theta_3}}{\proj_{V_0}d} + \inner{g_{\theta_3}}{\proj_{V^\perp_0}d} \\
&\ge a_3\norm{\proj_{V_0}d},
\ea
\edm
it follows that $\norm{\proj_{V_0}d}\le{O}(\norm{d}^2)$, and hence
$|\inner{g_i}{d}|\le\norm{g_i}\cdot\norm{\proj_{V_0}d}\le{O}(\norm{d}^2)$ for every $i\in\cM(x^*)$.
This concludes the proof of the claim.

It is ready to prove the theorem. Since the step size $\alpha$
used in \textproc{BranchSelection} is small enough, by \propref{prop:comp-select-while-loop}, 
the set $\cC$ returned by this function satisfies $\cC\subseteq\cM(x^*)$.
Then there are two possible cases: 
Case 1. $\cC$ does not contain a complete basis of $\cM(x^*)$; Case 2. It contains a complete basis of $\cM(x^*)$.
In Case 1, using \obsref{obs:complete-basis}, we conclude that $\norm{\vc{\QPt{x}{\cC}}}>c$ for a constant $c>0$.
Then by the termination criterion of \textproc{BranchSelection}, the reduction of objective value for this iteration
must be greater than $\rho_0\alpha{c^2}$, which is a constant. In this case, the convergence is better than linear rate. 

We now analyze Case 2. First, using Taylor expansion the function value satisfies 
\beq
f(x) - f^*\le\inner{g_\theta}{d} + \frac{1}{2}L_\Theta\norm{d}^2,
\eeq
where $\theta$ is an active branch at the current point $x$ and $d=x-x^*$.
We consider two sub-cases.
 
Case 2.1. $\inner{g_\theta}{d}\ge{M}\norm{d}^2$
for some large enough constant $M>0$. Let $\cC^\prime$ be the set of branches 
after completing Step 2 in \textproc{GapReduction}. 
(Note that Step 2 can be finished in a bounded number of iterations due to the finiteness of $\cM(x^*)$.)
It follows that $\cC\subseteq\cC^\prime$ and $\cC^\prime\cap\cM(x-\gamma{s^*})$ is nonempty.  
Since $\cC$ contains a complete basis due to the hypothesis of Case 2, so does $\cC^\prime$. 
We can then apply \lemref{lem:gap-red} to the set $\cC^\prime$ and conclude that
the function-value reduction obtained from the \textproc{GapReduction} has the following lower bound: 
\bdm
f(x) - f(x-\gamma{s^*})\ge\gamma\inner{g_\theta}{d} - c_0\norm{d}^2,
\edm
for some constant $c_0\ge0$.
Then it follows that
\bdm
\ba
&\frac{f(x) - f(x-\gamma{s^*})}{f(x)-f^*}\ge\frac{\gamma\inner{g_\theta}{d} - c_0\norm{d}^2}{\inner{g_\theta}{d} + L_\Theta\norm{d}^2/2} 
\ge\frac{\gamma\inner{g_\theta}{d}/2}{\inner{g_\theta}{d} + L_\Theta\norm{d}^2/2}\ge\frac{\gamma}{4},
\ea
\edm
where we use the hypothesis that $M$ is large enough.

Case 2.2 $\inner{g_\theta}{d}\le{M}\norm{d}^2$. Using Claim 2, we conclude
that $|\inner{g_i}{d}|\le{O}(\norm{d}^2)$ for every $i\in\cM(x^*)$.
In particular, we have $\norm{\proj_Vd}\le{O}(\norm{d}^2)$ where $V=\sp\{g_i: i\in\cC\}$ by Claim 2. 
Let $d^\prime=\qpvc{x}{\cC}$.
Applying \lemref{lem:vec>sigma*d} gives $\norm{d^\prime}\ge\sigma\norm{d}$
for some constant $\sigma>0$. The function-value reduction obtained from the 
\textproc{JointGradientReduction} has the following lower bound:
\bdm
f(x) - f(x-\alpha{d^\prime}) \ge\rho_0\alpha\norm{d^\prime}^2\ge\rho_0\alpha\sigma^2\norm{d}^2.
\edm
Therefore, the reduction ratio is lower bounded by:
\bdm
\frac{f(x) - f(x-\alpha{d^\prime}) }{f(x)-f^*}\ge\frac{\rho_0\alpha\sigma^2\norm{d}^2}{\inner{g_\theta}{d} + L_\Theta\norm{d}^2/2}
\ge\frac{\rho_0\alpha\sigma^2\norm{d}^2}{M\norm{d}^2 + L_\Theta\norm{d}^2/2}\ge\frac{\rho_0\alpha\sigma^2}{M+L_\Theta/2}.
\edm
We therefore have proved that the reduction of function value in Case 2 is in linear rate.  
Combining Case 1 and Case 2 concludes the proof. 
\end{proof}

\begin{algorithm}
{\scriptsize 
\caption{\scriptsize An enhanced branch-information-driven gradient descent algorithm (EBIGD)}\label{alg:global_and_local_converg}
\begin{algorithmic}[1]
	\State{\textbf{Input}: an initial point $x_{\textrm{init}}$, and parameters $\rho_0,\alpha,\gamma\in(0,1)$.}
	\State{\textbf{Initialization}: Set $k\gets{0}$, $x_k{\gets}x_{\textrm{init}}$.}
	\While{\textbf{true}}\label{lin:outer-while}
		\State{Step 1. Let $[\cC,d]=\Call{BranchSelection}{x_k,\alpha,\rho_0}$.} \label{lin:compute_d}
		\State{Step 2. Let $[y, reduction]=\Call{JointGradientReduction}{x_k,\alpha,d,\rho_0}$.}  \label{lin:call_GapRed}
		\State{Step 3. Let $[y^\prime, reduction^\prime]=\Call{GapReduction}{x_k,\gamma,\cC}$}
		\State{\textbf{If} $reduction\ge{reduction}^\prime$, set $x_{k+1}\gets{y}$; \textbf{else} set $x_{k+1}\gets{y^\prime}$.}
		\State{Let $k\gets{k+1}$.}
	\EndWhile
\end{algorithmic}
\vspace{-3pt}\noindent\makebox[\linewidth]{\rule{\linewidth}{0.4pt}}
\begin{algorithmic}[1]
\Function{BranchSelection}{$x,\alpha,\rho_0$}
	\State{Set $i\gets1$, let $\theta_1$ be an active branch at $x$ and set $\cC_1=\{\theta_1\}$.}
	\While{$\rho<\rho_0$}
		\State{Compute $d_i=\vc{\QPt{x}{\cC_i}}$.}
		\State{\textbf{If} $\norm{d_i}=0$, \textbf{break}.}
		\State{Compute $\rho\gets[f(x) - f(x-\alpha{d}_i)]/(\alpha\norm{d_i}^2)$}
		\State{Let $\theta_{i+1}$ be an active branch at $x-\alpha{d}_i$.} \label{lin:comp_select_i+1}
		\If{$\theta_{i+1}\in\tM(x)$ (given \assref{ass:verify-theta-in-tM})}
			\State{Set $\cC_{i+1}=\cC_i\cup\{\theta_{i+1}\}$ and $i\gets{i+1}$.}
		\EndIf
	\EndWhile
	\State{\textbf{return} [$\cC_i$, $d_i$]} \label{lin:comp_select_final_return}
\EndFunction
\end{algorithmic}
\vspace{-3pt}\noindent\makebox[\linewidth]{\rule{\linewidth}{0.4pt}}
\begin{algorithmic}[1]
\Function{JointGradientReduction}{$x,\alpha,d,\rho_0,$}
	\If{$\norm{d}>0$}
		\State{Set $y=x-\alpha{d}$ and $reduction=f(x)-f(x-\alpha{d})$.}
	\Else
		\State{Set $y=x$ and $reduction=0$.}
	\EndIf
	\State{\textbf{return} $[y, reduction]$.}
\EndFunction
\end{algorithmic}
\vspace{-3pt}\noindent\makebox[\linewidth]{\rule{\linewidth}{0.4pt}}
\begin{algorithmic}[1]
\Function{GapReduction}{$x,\gamma,\cC$}
	\State{Step 1. Solve the following direction-recovery convex program and let $(z^*, s^*)$ be an optimal solution:}
	\bdm
	\ba
	&\min_{z\in\mR,s\in\mR^n}\; \sum_{i\in\cC}(z+\inner{\nabla f_i(x)}{s} - f_i(x))^2 + \norm{s}^4.
	\ea
	\edm 
	\State{Step 2. If $\cM(x-\gamma{s^*})\cap\cC=\emptyset$, then select an arbitrary $\theta\in\cM(x-\gamma{s^*})$,}
	\State{set $\cC\gets\cC\cup\{\theta\}$ and go to Step 1.}
	\State{Step 3.}
	\If{$f(x)-f(x-\gamma{s^*})\ge0$}
		\State{Set $y=x-\gamma{s^*}$ and $reduction=f(x)-f(x-\gamma{s^*})$.}
	\Else
		\State{Set $y=x$ and $reduction=0$.}
	\EndIf
	\State{\textbf{return} $[y, reduction]$.}
\EndFunction
\end{algorithmic}
}
\end{algorithm}

\section{Numerical investigation}\label{sec:num-invest}
A practical implementation of the branch-information-driven gradient descent method is   
is given in \algoref{alg:practical-JGD}, which has been used in the numerical study.
This implementation mostly aligns with the gradient sampling algorithm developed in 
\cite{lewis2005-grad-sampling-nonsmooth-nonconvex-opt}. The difference is on the 
bundle points collection and line search. Specifically, in each iteration, 
the bundle points used in \algoref{alg:practical-JGD} for gradient computation are 
representative points of visited branches within a ball, 
whereas those in \cite{lewis2005-grad-sampling-nonsmooth-nonconvex-opt} are randomly sampled within a ball.
The line search in \algoref{alg:practical-JGD} employs a custom analysis tailored for this work, 
whereas that in \cite{lewis2005-grad-sampling-nonsmooth-nonconvex-opt} 
relies on the standard Armijo condition, which is less effective for piecewise-differentiable objectives.  
We consider 9 standard test problems from \cite{haarala2004-mem-bundle-nonsmooth-opt}
with 6 different dimensions: $n=25, 50, 75, 100, 125, 150, 175, 200$ in our experiments.
Their definition and properties are summarized in \tabref{tab:test_prob}.

We compare the performance of our branch-information-driven gradient descent method  
with other three methods for non-smooth optimization developed in literature. 
The methods and their notations are listed as follows:
\begin{itemize}
	\item[-] BIGD: the branch-information-driven gradient descent method (\algoref{alg:practical-JGD})
	 with parameters $\epsilon_0=0.1$, $\nu_0=1.0\times10^{-3}$, $\gamma=0.5$, 
	 $\epsilonopt=1.0\times10^{-5}$, $\nuopt=1.0\times10^{-4}$, $\theta_\epsilon=0.1$, $\theta_\nu=0.9$
	 and $\rho=1.0\times10^{-2}$
	\item[-] GS: the gradient sampling method  \cite{lewis2005-grad-sampling-nonsmooth-nonconvex-opt}
	\item[-] TRB: a trust-region bundle method \cite{monjezi2023-bundle-trust-region-locally-liptz-func}
	\item[-] QNS: a quasi-Newton method with neighorhood sampling \cite{curtis2015-quasi-newton-nonconvex-nonsmooth}.
\end{itemize}
For the methods GS, TRB and QNS, the parameters are selected following the corresponding references.   
\algoref{alg:practical-JGD} and all algorithms used in the comparison 
are implemented in Julia, gradient calculation is via the automatic differentiation package ForwardDiff \cite{forward-diff},
and the solver for quadratic programs called in each algorithm is OSQP \cite{osqp}. 
The code can be accessed via the following link:
\centerline{\scriptsize \url{https://github.com/luofqfrank/CID-JGD.git}}
for reproduction. All computations were performed on a system equipped with a 2.3 GHz Quad-Core Intel Core i7 CPU and 32GB of memory.
The computational time limit per problem instance is 5 minutes. 
Overall, the experiments results reveal that the BIGD method has the following three advantages compared to other methods:
\begin{itemize}
	\item[-] Reduced number of function and gradient evaluations:
		    This results from the efficient representation of visited branches, particularly compared to sampling-based methods.
	\item[-] Faster quadratic program (QP) solving per iteration (especially for high-dimensional problems): 
		    Each bundle point in the QP uniquely represents a visited branch near the current iterate.
		    The QP size depends only on the number of local branches—not on dimension or sample size.
	\item[-] More effective and informative line-search: This property is implied by \propref{prop:(f-f)/(ad^2)}. If a trial point belongs to a 
		    branch $\theta$ that is included in the set $\cT$ bundle point the search direction computed based on $\cT$ has incorporated the 
		    information of $\nabla{f_\theta}(x^k)$ which can lead to a quality decrease of the objective. If $\theta\notin\cT$ 
		    $\cT$ can be updated to include $\theta$, refining the QP and search direction until the condition $\theta\in\cT$ is satisfied.
\end{itemize}
It is worth to remark that utilization of branch information can be integrated into any existing nonsmooth optimization method that is
based on neighborhood sampling and bundle points to benefit from the above three advantages when applied to encodable piecewise-smooth
functions. The BIGD method is just an example demonstrating the power of using the branch information.   

\tabref{tab:comp-perf-preset} gives the comparison of the four methods
on the computational time and final objective values upon termination. 
The initial points used in the comparison are specified in \cite{haarala2004-mem-bundle-nonsmooth-opt}.
For each problem instance, if the method terminates within 5 mins, the computational time is reported.
Similarly, \tabref{tab:comp-perf-rand} gives the comparison under randomly generated initial points (but consistent used for all methods).
For both cases, the BIGD are not terminated by the time limit in only 2$\sim$3 instances of Chained\_LQ, while there are many more non-terminated
instances for the other three methods. Within terminated instances, the computational time of BIGD is systematically lower than the other three methods. 
When the termination criterion is met, the BIGD method consistently obtains an optimality gap $f-f^*$ at the scale of $10^{-5}\sim10^{-9}$,
which is superior to the GS and TRB methods. For the setting of random initial points (\tabref{tab:comp-perf-rand}), 
while the QNGS obtains better optimality gap than the BIGD for several problem instances, it often takes longer time and more iterations. 
For the setting of recommended initial points (\tabref{tab:comp-perf-preset}), the BIGD is systematically superior to the QNGS
in computational time, number of iterations and optimality gap. In particular, the QNGS does not perform well 
on some higher dimensional problem instances, i.e., $n\ge75$. 
Given that BIGD is a first-order method while QNGS is a second-order method, 
these results further demonstrate the remarkable advantage of incorporating branch information in optimization. 
\figref{fig:obj_dec} is the visualization of the optimality gap versus the computational time corresponding
to the four methods, plotted for 8 selected problem instances of dimension $n=200$. The figure shows that the BIGD
leads to the fastest decrease of objective values in the last 6 problem instances, while the BIGD is on par with the QNGS in
the first 2 problem instances. 

\tabref{tab:time-breakdown} provides a breakdown analysis of average computational time spending on two key steps of 
an iteration: function value and gradient evaluation, and solving of quadratic program, 
compared among the BIGD, GS and QNGS methods. The comparison between the BIGD and GS methods
clearly shows the three advantages of using the branch information: The BIGD spends 1$\sim$2 magnitude 
less computational time on function evaluation, solving the quadratic program and execution of each iteration.
The BIGD also shows systematically advantage over the QNGS on the iteration time, function evaluation time
and QP-solving time with a few exceptions, but the advantage is not as big as when compared with the GS method.
This similarity arises because BIGD and GS share identical procedural structures and both are first-order methods, 
implying that BIGD's advantage stems solely from its incorporation of branch information. 
In contrast, QNGS employs a second-order approach with BFGS updates, 
which inherently provides more sophisticated optimization capabilities than first-order methods. 
Moreover, QNGS selectively performs gradient sampling only when the objective function fails to achieve sufficient decrease. 
Consequently, the benefits of branch information utilization in BIGD are counterbalanced by QNGS's advanced algorithmic framework. 

We can get more insights on the mechanism of the BIGD method by investigating the dynamics of the bundle size $|\cT|$ 
which is the number of effective branches used in the QP \eqref{eqn:imp-alg-QP} to generate a search direction,  
and the number $|\Thetadisc|$ of visited branches in the algorithm. 
\figref{fig:act_branch} shows the plots of $|\cT|$ and $|\Thetadisc|$ versus the iteration index for 7 problem instances 
of dimension $n=200$. In each plot, the blue line and red line respectively represent the number of
effective branches used in the QP at each iteration and number of visited branches by the end of each iteration,
with the scale of the two lines displayed on the left and right axis respectively.
We observed that the number of effective branches fluctuates periodically within a certain range which is
much smaller than and does not increase with the accumulated number of visited branches, indicating that
the number of effective branches within the neighborhood of the current point is relatively stable. 
The periodic pattern of fluctuation reveals a cycle of neighborhood exploration and progressing. 
During the increasing phase of $|\cT|$, the algorithm explored the current neighborhood with small step sizes 
to encounter and collect all the `blocking' branches 
(branches that restrict the direction of making a large enough step) in several iterations. 
At a moment, all blocking branches were collected and included in $\cT$,
an efficient search-direction was identified, and a substantial decrease of objective 
was achieved, which corresponds to the sudden drop of $|\cT|$, concluding the progressing phase of a cycle. 

\section{Conclusion}
This paper demonstrates the effectiveness of functional encoding and branch information 
in designing algorithms for piecewise-smooth optimization. 
Numerical results provide strong evidence that incorporating branch information reduces 
the computational complexity of solving quadratic programs for search direction generation 
while enhancing the efficiency and informativeness of line-search procedures. 
Future research could explore the integration of branch information into existing 
nonsmooth optimization methods to further improve their performance.
Additionally, methodological extensions to different problem settings warrant further investigation. 
For instance, if certain branch functions take a stochastic form, such as $f_\theta(x)=\E_{\xi}[\psi(x,\xi)]$,
the evaluation of function values and gradients becomes stochastic. 
This stochasticity introduces new challenges and opportunities, 
potentially enriching both algorithmic design and convergence analysis.

\section*{Declarations}
\textbf{Data availability} The data used in this work is from an open-source dataset. \newline 
\textbf{Code availability} The code for this work is open access. \newline
\textbf{Conflict of interest} The author has no conflict of interest to declare that are relevant to the content of
this article. \newline
\textbf{Funding} No funding was received for this work. \newline

\begin{algorithm}
{\scriptsize 
\caption{\scriptsize A practical implementation for the branch-information-driven gradient descent method}\label{alg:practical-JGD}
\begin{algorithmic}
	\State{\textbf{Input}: initial point $x^0$, initial exploration radius $\epsilon_0\in(0,\infty)$, 
	initial stationarity target $\nu_0\in(0,\infty)$, 
	linear search parameters $\gamma\in(0,1)$, 
	termination tolerances $(\epsilonopt,\nuopt)\in(0,\infty)\times(0,\infty)$,
	reduction factors $(\theta_\epsilon,\theta_\nu)\in(0,1)\times(0,1)$,
	and line-search threshold $\rho_0$.}
	\State{\textbf{Initialization}: Set $\Thetadisc\gets\{\theta\}$ for an arbitrary $\theta\in\cM(x^0)$ and $z_\theta\gets{x^0}$}
	\For{$k\in\mathbb{N}$}
		\State{Let $\cT=\Thetadisc\cap\mB(x^k,\epsilon_k)$.}
		\State{Let $g^k=\Call{GradientComputation}{\cT}$.}
		\State{\textbf{If} $\|g^k\|\le\nuopt$ and $\epsilon_k\le\epsilonopt$ \textbf{then} terminate.}
		\If{$\|g^k\|\le\nu_k$}
			\State{Set $\nu_{k+1}\gets\theta_\nu\nu_k$, $\epsilon_{k+1}\gets\theta_\epsilon\epsilon_k$.}
		\Else
			\State{Set $\nu_{k+1}\gets\nu_k$ and $\epsilon_{k+1}\gets\epsilon_k$.}
			\State{Let $\alpha=\Call{LineSearch}{x^k,g^k,\gamma,\rho_0}$.}
			\State{Let $x^{k+1}\gets{x^k}-\alpha{g^k}/\|g^k\|$.}
		\EndIf
	\EndFor
\end{algorithmic}
\vspace{-5pt}\noindent\makebox[\linewidth]{\rule{\linewidth}{0.4pt}}
	\begin{algorithmic}[1]
		\Function{GradientComputation}{$\cT$}
					\State{Solve the following quadratic program:}
					\beq\label{eqn:imp-alg-QP}
					\min_\lambda \norm{\sum_{\theta\in\cT}\lambda_\theta\nabla f_\theta(z_\theta)}^2
					\;\textrm{ s.t. } \sum_{\theta\in\cT}\lambda_\theta = 1,\;\lambda_\theta\ge 0\;\forall\theta\in\cT.
					\eeq
					\State{Let $\lambda^*$ be the optimal solution of the above problem, and let 
					$g^*=\sum_{\theta\in\cT}\lambda^*_\theta\nabla f_\theta(z_\theta)$.}
				\State{\textbf{return} $g^*$.}
		 \EndFunction
	\end{algorithmic}
}
\vspace{-5pt}\noindent\makebox[\linewidth]{\rule{\linewidth}{0.4pt}}
	\begin{algorithmic}[1]
		\Function{LineSearch}{$x,g,\gamma,\rho_0$}
			\State{Set $\alpha\gets1$ and $flag\gets 0$. Let $d=g/\norm{g}$.}
					\While{$flag=0$}
						\State{Compute the following ratio: $\rho = [f(x) - f(x-\alpha{d})]/(\alpha\norm{g})$.}
						\If{$\rho\ge\rho_0$}
							\State{\Call{BranchPointUpdate}{$x-\alpha{d}$, $x-\alpha{d}$}.}
							\State{Set $flag\gets 1$.}
						\Else
							\State{\Call{BranchPointUpdate}{$x$, $x-\alpha{d}$}.}
							\State{Set $\alpha\gets\gamma\alpha$.}
						\EndIf
					\EndWhile	
			\State{\textbf{return} $\alpha$.}
		\EndFunction
	\end{algorithmic}
\vspace{-5pt}\noindent\makebox[\linewidth]{\rule{\linewidth}{0.4pt}}
	\begin{algorithmic}[1]
		\Procedure{BranchPointUpdate}{$x$, $\xtr$}
			\For{$\theta\in\cM(\xtr)$}
				\If{there does not exist $\theta$ in $\Thetadisc$}
					\State{Set $\Thetadisc\gets\Thetadisc\cup\{\theta\}$ and $z_\theta\gets\xtr$.}
				\ElsIf{$\norm{\xtr-x}<\norm{z_\theta-x}$}
					\State{Set $z_\theta\gets\xtr$.}
				\EndIf
			\EndFor
		\EndProcedure
	\end{algorithmic}
\end{algorithm}

\begin{table}
\rotatebox{90}{
\begin{minipage}{\textheight}
	\caption{Test problems from \cite{haarala2004-mem-bundle-nonsmooth-opt} and their properties}\label{tab:test_prob}
	{\scriptsize
	\renewcommand{\arraystretch}{3}
	\begin{tabular}{cccc}
	\toprule
	function name & analytical form & optimal objective & convexity  \\
	\midrule
	 gen\_MAXQ & $\max_{1\le{i}\le{n}}\;x^2_i$ & 0.0 & + \\
	gen\_MXHILB & $\max_{1\le{i}\le{n}}\left|\frac{x_i}{i+j-1}\right|$ & 0.0 & +  \\
	Chained\_LQ & $\sum^{n-1}_{i=1}\max\{-x_i-x_{i+1}, -x_i-x_{i+1}+(x^2_i+x^2_{i+1}-1)\}$ & $-\sqrt{2}(n-1)$ & +  \\
	Chained\_CB3\_I & $\sum^{n-1}_{i=1}\{x^4_i+x^2_{i+1}, (2-x_i)^2+(2-x_{i+1})^2, 2e^{-x_i+x_{i+1}} \}$ & $2(n-1)$ & + \\
	Chained\_CB3\_II & $\sum^{n-1}_{i=1}\max\{x^2_i+(x_{i+1}-1)^2+x_{i+1}-1, -x^2_i-(x_{i+1}-1)^2+x_{i+1}+1\}$ & $2(n-1)$ & + \\
	num\_active\_faces & $\max_{1\le{i}\le{n}}\left\{\log\left|\sum^n_{i=1}x_i\right|, \log|x_i|  \right\}$ & 0.0 & -  \\
	brown\_func2 & $\sum^{n-1}_{i=1}\big(|x_i|^{x^2_{i+1}+1} + |x_{i+1}|^{x^2_i+1} \big)$  & 0.0 & - \\
	Chained\_Crescent\_I & $\max\left\{\sum^{n-1}_{i=1}(x^2_i+(x_{i+1}-1)^2 + x_{i+1}-1), \sum^{n-1}_{i=1}(-x^2_i-(x_{i+1}-1)^2 + x_{i+1}+1) \right\}$ & 0.0 & - \\
	Chained\_Crescent\_II & $\sum^{n-1}_{i=1}\max\{x^2_i+(x_{i+1}-1)^2+x_{i+1}-1, -x^2_i-(x_{i+1}-1)^2 + x_{i+1}+1\}$ & 0.0 & - \\
	\bottomrule
	\end{tabular}
	}
\end{minipage}
}
\end{table}

{\newpage
\begin{table}[p]
\floatpagestyle{empty}
\centering
\rotatebox{90}{
\begin{minipage}{\textheight}
\vspace*{-3.25cm}
\caption{Computational performance with initial points specified in \cite{haarala2004-mem-bundle-nonsmooth-opt} }
\label{tab:comp-perf-preset}
\begin{tabular}{cr|rrc|rrc|rrc|rrc}
\hline\hline
& & \multicolumn{3}{c|}{BIGD} & \multicolumn{3}{c|}{GS} & \multicolumn{3}{c|}{TRB} & \multicolumn{3}{c}{QNGS} \\
\hline
problem	&	dim	&	time(s)	&	iter	&	obj	&	time(s)	&	iter	&	obj	&	time(s) &	iter	&	obj	&	time(s)	&	iter	&	obj	\\
\hline
gen\_MAXQ	&	25	&	1.33	&	307	&	2.74e-9	&	1.00	&	145	&	6.06e-9	&	107.00	&	60	&	8.88e-1	&	4.50	&	265	&	2.44e-9	\\
	&	50	&	2.89	&	670	&	4.43e-9	&	17.22	&	2240	&	7.98e-9	&	80.85	&	425	&	1.17e-6	&	43.17	&	4571	&	2.27e-9	\\
	&	75	&	4.77	&	1046	&	4.93e-9	&	-	&	30596	&	1.00e+0	&	-	&	68	&	6.73e-1	&	-	&	29186	&	1.00e+0	\\
	&	100	&	7.03	&	1511	&	4.62e-9	&	-	&	23176	&	1.00e+0	&	294.61	&	247	&	2.75e-1	&	-	&	19238	&	1.00e+0	\\
	&	150	&	16.65	&	2786	&	6.57e-9	&	-	&	14134	&	1.00e+0	&	0.00	&	1	&	1.00e+0	&	-	&	18941	&	1.00e+0	\\
	&	200	&	24.74	&	4384	&	1.49e-8	&	-	&	10020	&	1.00e+0	&	0.00	&	1	&	1.00e+0	&	-	&	18712	&	1.00e+0	\\
\hline
gen\_MXHILB	&	25	&	0.46	&	110	&	3.38e-6	&	1.45	&	104	&	1.96e-6	&	22.18	&	592	&	1.23e-4	&	-	&	12380	&	9.19e-10	\\
	&	50	&	0.84	&	176	&	2.45e-6	&	14.57	&	175	&	1.17e-6	&	57.71	&	960	&	8.05e-4	&	3.56	&	532	&	2.25e-8	\\
	&	75	&	0.90	&	172	&	1.99e-6	&	70.46	&	230	&	3.15e-6	&	59.37	&	920	&	8.49e-4	&	-	&	4839	&	4.49e-8	\\
	&	100	&	1.76	&	273	&	1.15e-6	&	135.50	&	205	&	7.44e-6	&	222.16	&	1902	&	2.86e-4	&	-	&	2427	&	1.44e-8	\\
	&	150	&	2.09	&	265	&	2.59e-6	&	-	&	120	&	8.76e-3	&	-	&	2002	&	1.20e-3	&	-	&	2575	&	1.82e-8	\\
	&	200	&	4.17	&	421	&	4.30e-6	&	-	&	56	&	1.67e-1	&	-	&	1841	&	3.33e-3	&	-	&	2724	&	2.93e-8	\\
\hline
Chained\_LQ	&	25	&	0.98	&	116	&	2.35e-6	&	11.01	&	752	&	2.13e-6	&	-	&	7117	&	7.88e-1	&	5.08	&	352	&	4.17e-7	\\
	&	50	&	5.67	&	234	&	3.70e-6	&	163.84	&	2090	&	3.18e-6	&	-	&	3289	&	2.78e-1	&	-	&	7878	&	1.64e-5	\\
	&	75	&	25.36	&	341	&	8.89e-6	&	-	&	1143	&	2.32e-1	&	52.13	&	572	&	7.63e-1	&	-	&	6039	&	1.07e-4	\\
	&	100	&	98.87	&	476	&	4.88e-6	&	-	&	499	&	8.76e-1	&	14.97	&	203	&	6.04e-1	&	-	&	2954	&	3.66e-2	\\
	&	150	&	-	&	104	&	7.51e-2	&	-	&	127	&	1.04e+0	&	13.50	&	193	&	2.77e-1	&	-	&	2062	&	2.18e-1	\\
	&	200	&	-	&	322	&	1.93e-3	&	-	&	58	&	1.19e+0	&	13.13	&	212	&	1.91e-1	&	-	&	1494	&	5.92e-2	\\
\hline
Chained\_CB3\_I	&	25	&	1.25	&	109	&	2.81e-5	&	16.73	&	1167	&	1.89e-5	&	-	&	3191	&	1.36e+0	&	1.19	&	211	&	1.99e-5	\\
	&	50	&	7.00	&	185	&	4.00e-5	&	268.99	&	3413	&	2.37e-5	&	30.01	&	272	&	1.24e+0	&	164.16	&	4562	&	4.20e-7	\\
	&	75	&	8.45	&	177	&	7.09e-5	&	-	&	1134	&	9.27e-1	&	6.29	&	131	&	6.07e-1	&	-	&	4902	&	1.70e-5	\\
	&	100	&	10.36	&	171	&	5.95e-5	&	-	&	508	&	1.76e+0	&	15.29	&	139	&	1.56e+0	&	-	&	3220	&	2.21e-5	\\
	&	150	&	22.39	&	195	&	7.87e-5	&	-	&	111	&	1.90e+0	&	30.51	&	156	&	1.22e+0	&	-	&	1678	&	1.58e+1	\\
	&	200	&	274.03	&	375	&	7.46e-5	&	-	&	52	&	2.40e+0	&	10.46	&	171	&	1.09e+0	&	-	&	1623	&	1.24e+1	\\
\hline
Chained\_CB3\_II	&	25	&	0.40	&	65	&	1.34e-7	&	2.14	&	85	&	7.11e-6	&	2.40	&	71	&	2.10e-7	&	-	&	5030	&	3.17e-4	\\
	&	50	&	0.86	&	116	&	1.86e-7	&	54.11	&	311	&	1.44e-6	&	3.23	&	98	&	3.25e-7	&	-	&	3741	&	1.34e-4	\\
	&	75	&	0.67	&	90	&	2.21e-7	&	219.43	&	461	&	9.93e-6	&	4.28	&	111	&	4.64e-7	&	-	&	3556	&	9.80e-5	\\
	&	100	&	0.98	&	110	&	3.90e-7	&	274.83	&	291	&	2.49e-5	&	7.09	&	122	&	8.23e-7	&	-	&	4160	&	3.54e-4	\\
	&	150	&	0.36	&	59	&	3.38e-6	&	-	&	109	&	1.53e-4	&	10.30	&	149	&	7.00e-7	&	-	&	5594	&	2.69e-4	\\
	&	200	&	0.48	&	60	&	4.75e-7	&	-	&	47	&	4.88e-3	&	8.75	&	168	&	4.50e-7	&	-	&	1784	&	2.29e-5	\\
\hline
num\_active\_faces	&	25	&	0.06	&	13	&	1.12e-7	&	0.18	&	9	&	4.88e-15	&	0.64	&	56	&	2.43e-6	&	-	&	6787	&	5.05e-5	\\
	&	50	&	0.12	&	22	&	1.18e-7	&	2.81	&	24	&	1.32e-5	&	1.39	&	103	&	1.68e-7	&	-	&	2167	&	8.62e-5	\\
	&	75	&	0.13	&	26	&	4.38e-6	&	9.83	&	23	&	1.21e-5	&	1.70	&	116	&	2.07e-6	&	-	&	79	&	4.47e-3	\\
	&	100	&	0.09	&	18	&	1.49e-7	&	9.96	&	14	&	1.38e-14	&	3.36	&	160	&	1.58e-6	&	-	&	1631	&	1.31e-4	\\
	&	150	&	0.14	&	24	&	9.42e-6	&	70.20	&	24	&	9.14e-6	&	6.13	&	215	&	4.09e-6	&	-	&	1702	&	1.83e-4	\\
	&	200	&	0.16	&	30	&	1.60e-7	&	172.81	&	26	&	5.11e-8	&	6.75	&	305	&	3.28e-6	&	11.57	&	177	&	1.94e-5	\\
\hline
brown\_func2	&	25	&	0.25	&	44	&	2.43e-5	&	3.06	&	203	&	4.74e-6	&	1.73	&	83	&	7.14e-7	&	-	&	11963	&	4.38e-4	\\
	&	50	&	0.33	&	56	&	5.63e-6	&	102.17	&	1256	&	4.30e-6	&	3.17	&	122	&	9.75e-7	&	-	&	6802	&	9.25e-4	\\
	&	75	&	0.33	&	56	&	2.64e-5	&	-	&	1094	&	5.55e-1	&	2.89	&	116	&	7.35e-7	&	-	&	5624	&	4.28e-3	\\
	&	100	&	0.49	&	71	&	4.95e-5	&	-	&	502	&	8.36e-1	&	4.97	&	132	&	2.35e-7	&	-	&	4404	&	9.39e-3	\\
	&	150	&	0.59	&	67	&	1.43e-5	&	-	&	123	&	9.42e-1	&	6.94	&	149	&	4.03e-7	&	-	&	3269	&	2.14e-2	\\
	&	200	&	0.74	&	69	&	2.29e-5	&	-	&	52	&	9.67e-1	&	6.87	&	177	&	4.39e-7	&	-	&	2449	&	3.37e-2	\\
\hline
Chained\_Crescent\_I	&	25	&	0.21	&	47	&	5.49e-7	&	1.39	&	64	&	9.77e-6	&	1.98	&	104	&	9.56e-7	&	-	&	2426	&	7.50e-4	\\
	&	50	&	0.25	&	53	&	1.57e-8	&	12.85	&	90	&	5.47e-6	&	4.88	&	181	&	9.29e-7	&	276.66	&	26006	&	1.22e-5	\\
	&	75	&	0.29	&	55	&	1.06e-7	&	31.69	&	73	&	1.21e-5	&	4.83	&	169	&	1.40e-6	&	-	&	1906	&	2.40e-2	\\
	&	100	&	0.29	&	54	&	7.82e-8	&	96.55	&	106	&	3.90e-6	&	12.15	&	248	&	9.59e-7	&	-	&	1718	&	1.66e-3	\\
	&	150	&	0.35	&	67	&	1.04e-6	&	-	&	94	&	1.04e-2	&	14.34	&	235	&	1.14e-6	&	-	&	1544	&	2.79e-2	\\
	&	200	&	0.52	&	91	&	1.28e-7	&	-	&	45	&	1.16e+1	&	13.14	&	269	&	4.54e-7	&	-	&	1238	&	2.30e-3	\\
\hline
Chained\_Crescent\_II	&	25	&	0.24	&	52	&	1.28e-6	&	8.03	&	459	&	5.19e-6	&	-	&	10221	&	6.22e+0	&	-	&	11222	&	1.34e-3	\\
	&	50	&	0.56	&	96	&	1.12e-5	&	262.32	&	2985	&	4.06e-6	&	-	&	8130	&	1.86e+1	&	-	&	6872	&	1.23e-1	\\
	&	75	&	0.47	&	72	&	6.04e-6	&	-	&	946	&	1.24e-1	&	-	&	6285	&	2.70e+1	&	-	&	4758	&	3.59e-1	\\
	&	100	&	0.65	&	105	&	1.92e-5	&	-	&	478	&	1.45e-1	&	-	&	4478	&	3.79e+1	&	-	&	3835	&	3.59e-1	\\
	&	150	&	0.77	&	112	&	5.94e-6	&	-	&	126	&	2.91e-1	&	-	&	3085	&	5.47e+1	&	-	&	2703	&	3.59e-1	\\
	&	200	&	0.80	&	90	&	8.67e-6	&	-	&	48	&	2.90e-1	&	-	&	5526	&	7.81e+1	&	-	&	2161	&	3.59e-1	\\
\hline
\end{tabular}
\end{minipage}
}
\end{table}
}

{\newpage
\begin{table}[p]
\floatpagestyle{empty}
\centering
\rotatebox{90}{
\begin{minipage}{\textheight}
\vspace*{-3.75cm}
\caption{Computational performance with random initial points}
\label{tab:comp-perf-rand}
\begin{tabular}{cr|rrc|rrc|rrc|rrc}
\hline\hline
& & \multicolumn{3}{c|}{BIGD} & \multicolumn{3}{c|}{GS} & \multicolumn{3}{c|}{TRB} & \multicolumn{3}{c}{QNGS} \\
\hline
problem	&	dim	&	time(s)	&	iter	&	obj	&	time(s)	&	iter	&	obj	&	time(s) &	iter	&	obj	&	time(s)	&	iter	&	obj	\\
\hline
gen\_MAXQ	&	25	&	0.79	&	221	&	2.24e-9	&	1.96	&	156	&	5.84e-9	&	16.90	&	396	&	7.75e-1	&	1.87	&	241	&	3.77e-9	\\
	&	50	&	1.41	&	385	&	3.33e-9	&	22.21	&	255	&	5.26e-9	&	64.62	&	1462	&	7.54e-1	&	2.03	&	423	&	3.75e-9	\\
	&	75	&	2.21	&	568	&	7.18e-9	&	53.94	&	329	&	7.09e-9	&	-	&	1884	&	8.63e-1	&	3.31	&	566	&	2.46e-9	\\
	&	100	&	3.53	&	913	&	5.18e-9	&	159.69	&	446	&	1.10e-8	&	26.06	&	181	&	8.64e-1	&	5.60	&	721	&	2.44e-9	\\
	&	150	&	5.34	&	1311	&	7.82e-9	&	-	&	255	&	8.28e-3	&	14.85	&	94	&	8.69e-1	&	9.31	&	1057	&	2.48e-9	\\
	&	200	&	15.29	&	3540	&	1.29e-8	&	-	&	145	&	1.52e-1	&	176.33	&	737	&	8.00e-1	&	11.46	&	1369	&	2.47e-9	\\
\hline
gen\_MXHILB	&	25	&	0.75	&	180	&	4.56e-6	&	3.67	&	142	&	3.81e-5	&	26.30	&	157	&	2.51e-4	&	-	&	10511	&	1.01e-6	\\
	&	50	&	0.52	&	110	&	1.36e-5	&	14.29	&	99	&	1.44e-5	&	40.87	&	269	&	2.84e-4	&	-	&	8529	&	4.11e-7	\\
	&	75	&	1.19	&	227	&	1.43e-5	&	49.95	&	160	&	1.36e-5	&	24.06	&	167	&	2.24e-4	&	-	&	6757	&	1.17e-6	\\
	&	100	&	1.25	&	250	&	2.04e-5	&	127.85	&	190	&	2.38e-5	&	55.34	&	252	&	1.55e-4	&	-	&	3276	&	1.66e-6	\\
	&	150	&	0.96	&	167	&	4.79e-5	&	-	&	116	&	1.39e-4	&	-	&	490	&	8.19e-4	&	-	&	5589	&	2.91e-7	\\
	&	200	&	2.22	&	327	&	6.02e-6	&	-	&	52	&	9.40e-4	&	-	&	608	&	7.26e-4	&	-	&	2765	&	3.41e-6	\\
\hline
Chained\_LQ	&	25	&	2.19	&	205	&	4.00e-6	&	7.02	&	452	&	1.93e-6	&	-	&	3436	&	1.40e+0	&	-	&	15253	&	1.71e-13	\\
	&	50	&	27.04	&	417	&	5.27e-6	&	111.67	&	1218	&	2.42e-6	&	-	&	3661	&	2.48e+0	&	-	&	7584	&	1.58e-12	\\
	&	75	&	205.40	&	642	&	4.77e-6	&	-	&	1101	&	4.26e-2	&	-	&	4076	&	3.45e+0	&	-	&	5843	&	1.02e-12	\\
	&	100	&	-	&	497	&	1.72e-4	&	-	&	524	&	1.15e-1	&	-	&	2867	&	4.20e+0	&	-	&	2952	&	1.07e-4	\\
	&	150	&	-	&	291	&	1.23e-2	&	-	&	125	&	4.64e-1	&	-	&	2453	&	6.14e+0	&	-	&	2455	&	1.58e-3	\\
	&	200	&	-	&	263	&	2.09e-2	&	-	&	48	&	5.82e-1	&	-	&	2423	&	7.41e+0	&	-	&	2219	&	9.88e-4	\\
\hline
Chained\_CB3\_I	&	25	&	2.13	&	136	&	3.74e-5	&	8.92	&	566	&	2.11e-5	&	-	&	3727	&	3.50e+0	&	2.06	&	304	&	9.15e-6	\\
	&	50	&	10.76	&	230	&	3.55e-5	&	179.31	&	2312	&	2.10e-5	&	-	&	3506	&	6.21e+0	&	105.90	&	3118	&	6.47e-6	\\
	&	75	&	18.70	&	269	&	5.57e-5	&	-	&	1168	&	5.31e-1	&	-	&	3706	&	1.36e+1	&	-	&	5080	&	1.18e-3	\\
	&	100	&	38.26	&	298	&	9.51e-5	&	-	&	540	&	1.15e+0	&	-	&	2704	&	1.60e+1	&	-	&	2872	&	2.42e-3	\\
	&	150	&	124.69	&	437	&	7.34e-5	&	-	&	126	&	1.94e+0	&	-	&	2139	&	2.47e+1	&	-	&	1779	&	1.00e-1	\\
	&	200	&	269.20	&	422	&	1.17e-4	&	-	&	55	&	1.50e+0	&	-	&	2477	&	2.98e+1	&	-	&	1592	&	1.11e-1	\\
\hline
Chained\_CB3\_II	&	25	&	0.28	&	80	&	1.33e-7	&	1.81	&	71	&	9.78e-6	&	-	&	3480	&	2.20e+0	&	0.57	&	93	&	6.70e-7	\\
	&	50	&	0.26	&	65	&	1.21e-7	&	19.06	&	111	&	4.61e-6	&	7.40	&	112	&	6.49e-7	&	-	&	5440	&	1.54e-9	\\
	&	75	&	0.36	&	80	&	1.78e-7	&	36.98	&	79	&	7.08e-6	&	9.25	&	135	&	2.87e-7	&	2.28	&	93	&	4.74e-10	\\
	&	100	&	0.35	&	89	&	1.45e-7	&	83.32	&	89	&	8.59e-7	&	12.71	&	146	&	5.36e-7	&	1.78	&	109	&	2.28e-10	\\
	&	150	&	0.44	&	91	&	3.75e-7	&	-	&	94	&	1.51e-6	&	19.93	&	188	&	5.03e-7	&	-	&	1599	&	1.02e-9	\\
	&	200	&	0.48	&	102	&	2.07e-7	&	-	&	46	&	4.79e-1	&	19.23	&	204	&	4.83e-7	&	3.83	&	113	&	9.66e-10	\\
\hline
num\_active\_faces	&	25	&	1.03	&	263	&	5.54e-6	&	3.53	&	178	&	6.12e-6	&	2.21	&	65	&	5.63e-1	&	256.40	&	68421	&	5.59e-1	\\
	&	50	&	2.46	&	550	&	5.25e-6	&	25.60	&	282	&	1.44e-5	&	2.24	&	50	&	5.70e-1	&	-	&	43564	&	3.05e-3	\\
	&	75	&	5.13	&	980	&	8.36e-6	&	91.99	&	421	&	1.17e-5	&	2.92	&	58	&	6.14e-1	&	-	&	54910	&	1.16e-3	\\
	&	100	&	8.95	&	1497	&	6.89e-6	&	220.15	&	553	&	1.41e-5	&	2.44	&	39	&	6.29e-1	&	10.95	&	1084	&	6.33e-7	\\
	&	150	&	27.18	&	3013	&	7.43e-6	&	-	&	268	&	1.54e-1	&	3.20	&	46	&	6.51e-1	&	18.60	&	1751	&	2.33e-7	\\
	&	200	&	95.65	&	6112	&	8.56e-6	&	-	&	141	&	4.17e-1	&	200.63	&	976	&	7.11e-7	&	-	&	35412	&	5.69e-2	\\
\hline
brown\_func2	&	25	&	0.17	&	54	&	3.27e-5	&	4.83	&	350	&	7.37e-6	&	4.90	&	102	&	7.20e-7	&	2.85	&	222	&	1.09e-5	\\
	&	50	&	0.39	&	75	&	4.26e-5	&	47.34	&	601	&	1.72e-6	&	-	&	5344	&	6.31e+0	&	11.26	&	652	&	1.39e-6	\\
	&	75	&	3.27	&	138	&	1.97e-5	&	-	&	1102	&	8.55e-2	&	-	&	2512	&	1.08e+1	&	-	&	5729	&	4.55e-10	\\
	&	100	&	5.91	&	137	&	1.45e-5	&	-	&	518	&	3.15e-1	&	-	&	1907	&	1.37e+1	&	-	&	4286	&	2.30e-9	\\
	&	150	&	96.21	&	343	&	3.07e-5	&	-	&	124	&	3.76e-1	&	-	&	1304	&	1.92e+1	&	-	&	2174	&	7.17e-4	\\
	&	200	&	9.93	&	135	&	6.05e-5	&	-	&	54	&	3.82e-1	&	-	&	920	&	4.07e-1	&	-	&	1006	&	2.32e-2	\\
\hline
Chained\_Crescent\_I	&	25	&	0.11	&	38	&	1.85e-7	&	1.05	&	57	&	7.93e-6	&	4.61	&	47	&	1.34e-6	&	-	&	1631	&	9.97e-5	\\
	&	50	&	0.17	&	50	&	9.75e-9	&	7.97	&	64	&	1.07e-5	&	5.23	&	56	&	1.40e-6	&	0.68	&	69	&	3.86e-9	\\
	&	75	&	0.13	&	41	&	1.53e-7	&	24.85	&	65	&	1.06e-5	&	9.28	&	94	&	1.12e-6	&	1.29	&	93	&	2.91e-9	\\
	&	100	&	0.18	&	53	&	2.44e-10	&	50.22	&	60	&	1.33e-6	&	8.12	&	71	&	4.11e-7	&	-	&	1092	&	3.41e-4	\\
	&	150	&	0.18	&	49	&	2.89e-8	&	226.13	&	75	&	1.02e-5	&	16.67	&	85	&	1.14e-6	&	2.67	&	88	&	2.31e-7	\\
	&	200	&	0.16	&	43	&	6.20e-7	&	-	&	46	&	7.53e-4	&	12.80	&	100	&	8.22e-7	&	-	&	1244	&	1.56e-11	\\
\hline
Chained\_Crescent\_II	&	25	&	0.26	&	75	&	1.06e-5	&	5.76	&	359	&	3.75e-6	&	-	&	2730	&	7.06e+0	&	-	&	13792	&	2.11e-15	\\
	&	50	&	0.49	&	92	&	1.16e-5	&	136.85	&	1599	&	3.86e-6	&	-	&	2320	&	1.20e+1	&	-	&	8668	&	8.23e-14	\\
	&	75	&	0.72	&	133	&	6.12e-6	&	-	&	1147	&	2.17e-2	&	-	&	2191	&	1.28e+1	&	-	&	5251	&	3.15e-3	\\
	&	100	&	0.45	&	101	&	3.32e-6	&	-	&	522	&	1.03e+0	&	-	&	1714	&	1.52e+1	&	-	&	2939	&	3.07e-3	\\
	&	150	&	99.89	&	347	&	3.50e-5	&	-	&	121	&	2.92e-1	&	-	&	1765	&	3.13e+1	&	-	&	2056	&	2.31e-2	\\
	&	200	&	11.83	&	215	&	1.88e-5	&	-	&	46	&	1.32e+0	&	-	&	1656	&	4.49e+1	&	-	&	1399	&	8.91e-2	\\
\hline
\end{tabular}
\end{minipage}
}
\end{table}
}

{\newpage
\begin{table}[p]
\floatpagestyle{empty}
\centering
\rotatebox{90}{
\begin{minipage}{\textheight}
\vspace*{-3.25cm}
\caption{Computational time per iteration}
\label{tab:time-breakdown}
\begin{tabular}{cr|ccc|ccc|ccc}
\hline\hline
& & \multicolumn{3}{c|}{BIGD} & \multicolumn{3}{c|}{GS}  & \multicolumn{3}{c}{QNGS} \\
\hline
problem	&	dim	&	avg. iter. time	&	avg. eval. time	&	avg. QP. time	&	avg. iter. time	&	avg. eval. time	&	avg. QP. time	&    avg. iter. time	&  avg. eval. time	&  avg. QP. time	\\
\hline
gen\_MAXQ	&	25	&	3.58e-3	&	2.01e-5	&	3.36e-3	&	1.27e-2	&	4.80e-5	&	1.25e-2	&	7.79e-3	&	2.30e-3	&	4.97e-3	\\
	&	50	&	3.68e-3	&	2.12e-5	&	3.53e-3	&	8.74e-2	&	1.66e-4	&	8.72e-2	&	4.81e-3	&	1.41e-4	&	4.54e-3	\\
	&	75	&	3.90e-3	&	2.76e-5	&	3.70e-3	&	1.64e-1	&	2.49e-4	&	1.64e-1	&	5.86e-3	&	3.22e-4	&	5.36e-3	\\
	&	100	&	3.87e-3	&	3.88e-5	&	3.62e-3	&	3.59e-1	&	4.12e-4	&	3.58e-1	&	7.78e-3	&	8.32e-4	&	6.05e-3	\\
	&	150	&	4.08e-3	&	4.67e-5	&	3.70e-3	&	1.26e+0	&	1.73e-3	&	1.25e+0	&	8.82e-3	&	1.18e-3	&	7.21e-3	\\
	&	200	&	4.32e-3	&	6.64e-5	&	3.82e-3	&	2.22e+0	&	1.78e-3	&	2.22e+0	&	8.38e-3	&	1.72e-3	&	6.28e-3	\\
\hline
gen\_MXHILB	&	25	&	4.18e-3	&	3.20e-5	&	3.99e-3	&	2.61e-2	&	7.48e-5	&	2.58e-2	&	3.03e-2	&	1.35e-3	&	2.87e-2	\\
	&	50	&	4.78e-3	&	4.20e-5	&	4.50e-3	&	1.46e-1	&	4.92e-4	&	1.45e-1	&	3.73e-2	&	1.56e-3	&	3.53e-2	\\
	&	75	&	5.26e-3	&	6.49e-5	&	4.70e-3	&	3.14e-1	&	2.05e-4	&	3.14e-1	&	4.72e-2	&	1.78e-3	&	4.45e-2	\\
	&	100	&	5.03e-3	&	7.62e-5	&	4.36e-3	&	6.76e-1	&	2.83e-4	&	6.76e-1	&	9.75e-2	&	4.20e-3	&	9.17e-2	\\
	&	150	&	5.76e-3	&	1.42e-4	&	4.61e-3	&	2.76e+0	&	5.03e-4	&	2.76e+0	&	5.72e-2	&	2.83e-3	&	5.17e-2	\\
	&	200	&	6.80e-3	&	2.27e-4	&	4.92e-3	&	6.22e+0	&	9.37e-4	&	6.22e+0	&	1.16e-1	&	6.49e-3	&	1.05e-1	\\
\hline
Chained\_LQ	&	25	&	1.08e-2	&	1.81e-4	&	1.00e-2	&	1.56e-2	&	4.99e-5	&	1.55e-2	&	2.08e-2	&	7.48e-4	&	1.99e-2	\\
	&	50	&	6.50e-2	&	8.13e-4	&	6.28e-2	&	9.18e-2	&	6.84e-5	&	9.16e-2	&	4.20e-2	&	1.22e-3	&	4.05e-2	\\
	&	75	&	3.20e-1	&	1.81e-3	&	3.16e-1	&	2.90e-1	&	1.12e-4	&	2.90e-1	&	5.46e-2	&	1.43e-3	&	5.28e-2	\\
	&	100	&	6.45e-1	&	3.12e-3	&	6.39e-1	&	6.10e-1	&	3.09e-4	&	6.10e-1	&	1.08e-1	&	2.62e-3	&	1.05e-1	\\
	&	150	&	1.10e+0	&	5.04e-3	&	1.10e+0	&	2.55e+0	&	6.64e-4	&	2.55e+0	&	1.30e-1	&	3.12e-3	&	1.27e-1	\\
	&	200	&	1.23e+0	&	6.06e-3	&	1.22e+0	&	6.66e+0	&	9.87e-4	&	6.66e+0	&	1.44e-1	&	3.72e-3	&	1.40e-1	\\
\hline
Chained\_CB3\_I	&	25	&	1.58e-2	&	2.82e-4	&	1.37e-2	&	1.58e-2	&	4.66e-5	&	1.56e-2	&	6.81e-3	&	2.53e-4	&	6.15e-3	\\
	&	50	&	4.70e-2	&	8.05e-4	&	4.45e-2	&	7.76e-2	&	7.37e-5	&	7.74e-2	&	3.40e-2	&	1.71e-3	&	3.16e-2	\\
	&	75	&	6.98e-2	&	1.13e-3	&	6.71e-2	&	2.73e-1	&	1.38e-4	&	2.73e-1	&	6.28e-2	&	3.51e-3	&	5.83e-2	\\
	&	100	&	1.29e-1	&	1.54e-3	&	1.26e-1	&	5.91e-1	&	3.03e-4	&	5.90e-1	&	1.11e-1	&	6.02e-3	&	1.03e-1	\\
	&	150	&	2.86e-1	&	3.04e-3	&	2.80e-1	&	2.53e+0	&	1.36e-3	&	2.53e+0	&	1.80e-1	&	9.81e-3	&	1.67e-1	\\
	&	200	&	6.39e-1	&	5.86e-3	&	6.29e-1	&	5.86e+0	&	9.43e-4	&	5.86e+0	&	2.01e-1	&	1.08e-2	&	1.86e-1	\\
\hline
Chained\_CB3\_II	&	25	&	3.53e-3	&	2.54e-5	&	3.39e-3	&	2.59e-2	&	6.43e-5	&	2.58e-2	&	6.20e-3	&	1.43e-4	&	5.96e-3	\\
	&	50	&	4.11e-3	&	3.78e-5	&	3.93e-3	&	1.73e-1	&	1.01e-4	&	1.73e-1	&	5.87e-2	&	1.30e-3	&	5.72e-2	\\
	&	75	&	4.59e-3	&	5.02e-5	&	4.35e-3	&	4.74e-1	&	1.18e-4	&	4.74e-1	&	2.48e-2	&	1.02e-2	&	1.44e-2	\\
	&	100	&	3.99e-3	&	6.38e-5	&	3.73e-3	&	9.47e-1	&	3.28e-4	&	9.46e-1	&	1.65e-2	&	7.85e-4	&	1.54e-2	\\
	&	150	&	4.89e-3	&	1.04e-4	&	4.23e-3	&	3.26e+0	&	7.10e-4	&	3.26e+0	&	2.00e-1	&	4.99e-3	&	1.94e-1	\\
	&	200	&	4.73e-3	&	1.25e-4	&	4.05e-3	&	6.98e+0	&	1.04e-3	&	6.98e+0	&	3.42e-2	&	2.69e-3	&	3.11e-2	\\
\hline
num\_active\_faces	&	25	&	3.93e-3	&	5.07e-7	&	3.49e-3	&	1.99e-2	&	2.14e-4	&	1.97e-2	&	3.75e-3	&	3.19e-5	&	3.42e-3	\\
	&	50	&	4.49e-3	&	5.56e-7	&	3.61e-3	&	9.11e-2	&	1.19e-4	&	9.09e-2	&	7.20e-3	&	7.90e-4	&	6.16e-3	\\
	&	75	&	5.24e-3	&	5.73e-7	&	3.65e-3	&	2.19e-1	&	1.59e-4	&	2.19e-1	&	5.57e-3	&	7.32e-5	&	4.74e-3	\\
	&	100	&	5.98e-3	&	6.50e-7	&	3.66e-3	&	3.99e-1	&	2.59e-4	&	3.98e-1	&	1.01e-2	&	1.02e-3	&	8.09e-3	\\
	&	150	&	9.02e-3	&	9.14e-7	&	3.76e-3	&	1.18e+0	&	5.21e-4	&	1.18e+0	&	1.06e-2	&	1.54e-3	&	8.36e-3	\\
	&	200	&	1.57e-2	&	1.40e-6	&	4.20e-3	&	2.24e+0	&	7.97e-4	&	2.24e+0	&	8.86e-3	&	2.73e-4	&	7.98e-3	\\
\hline
brown\_func2	&	25	&	3.27e-3	&	2.38e-7	&	3.04e-3	&	1.38e-2	&	3.45e-5	&	1.37e-2	&	1.29e-2	&	7.54e-4	&	1.19e-2	\\
	&	50	&	5.21e-3	&	5.51e-7	&	4.70e-3	&	7.89e-2	&	8.49e-5	&	7.88e-2	&	1.73e-2	&	1.57e-3	&	1.52e-2	\\
	&	75	&	2.39e-2	&	1.45e-6	&	2.28e-2	&	2.90e-1	&	1.15e-4	&	2.90e-1	&	5.57e-2	&	5.89e-3	&	4.93e-2	\\
	&	100	&	4.35e-2	&	1.94e-6	&	4.14e-2	&	6.18e-1	&	2.53e-4	&	6.17e-1	&	7.45e-2	&	9.16e-3	&	6.43e-2	\\
	&	150	&	2.81e-1	&	4.38e-6	&	2.77e-1	&	2.58e+0	&	5.32e-4	&	2.58e+0	&	1.47e-1	&	2.21e-2	&	1.23e-1	\\
	&	200	&	7.41e-2	&	1.55e-6	&	6.99e-2	&	5.92e+0	&	9.82e-4	&	5.92e+0	&	3.18e-1	&	5.17e-2	&	2.64e-1	\\
\hline
Chained\_Crescent\_I	&	25	&	2.93e-3	&	1.50e-5	&	2.84e-3	&	1.88e-2	&	7.46e-5	&	1.87e-2	&	1.96e-1	&	1.21e-3	&	1.95e-1	\\
	&	50	&	3.52e-3	&	2.62e-5	&	3.39e-3	&	1.26e-1	&	8.03e-5	&	1.26e-1	&	9.97e-3	&	4.86e-4	&	9.33e-3	\\
	&	75	&	3.17e-3	&	2.58e-5	&	3.04e-3	&	3.88e-1	&	1.03e-4	&	3.88e-1	&	1.40e-2	&	6.76e-4	&	1.31e-2	\\
	&	100	&	3.47e-3	&	3.09e-5	&	3.31e-3	&	8.51e-1	&	2.91e-4	&	8.51e-1	&	2.93e-1	&	3.12e-3	&	2.89e-1	\\
	&	150	&	3.82e-3	&	4.34e-5	&	3.62e-3	&	3.06e+0	&	7.17e-4	&	3.05e+0	&	3.07e-2	&	1.70e-3	&	2.86e-2	\\
	&	200	&	3.84e-3	&	6.20e-5	&	3.59e-3	&	6.92e+0	&	1.67e-3	&	6.91e+0	&	2.57e-1	&	5.71e-3	&	2.51e-1	\\
\hline
Chained\_Crescent\_II	&	25	&	3.48e-3	&	3.00e-5	&	3.29e-3	&	1.61e-2	&	3.65e-5	&	1.60e-2	&	2.30e-2	&	8.77e-4	&	2.20e-2	\\
	&	50	&	5.41e-3	&	6.36e-5	&	5.10e-3	&	8.56e-2	&	7.67e-5	&	8.55e-2	&	3.67e-2	&	1.21e-3	&	3.53e-2	\\
	&	75	&	5.44e-3	&	8.62e-5	&	4.91e-3	&	2.79e-1	&	1.13e-4	&	2.79e-1	&	6.08e-2	&	1.77e-3	&	5.87e-2	\\
	&	100	&	4.46e-3	&	6.92e-5	&	4.08e-3	&	6.12e-1	&	2.80e-4	&	6.12e-1	&	1.09e-1	&	2.65e-3	&	1.06e-1	\\
	&	150	&	2.89e-1	&	1.71e-3	&	2.85e-1	&	2.66e+0	&	6.28e-4	&	2.66e+0	&	1.56e-1	&	3.75e-3	&	1.51e-1	\\
	&	200	&	5.53e-2	&	8.04e-4	&	5.30e-2	&	6.98e+0	&	1.05e-3	&	6.98e+0	&	2.29e-1	&	5.25e-3	&	2.23e-1	\\
\hline
\end{tabular}
\end{minipage}
}
\end{table}
}

\begin{figure}[htbp]
    \centering
    \begin{subfigure}[b]{0.48\textwidth}
        \includegraphics[width=\textwidth]{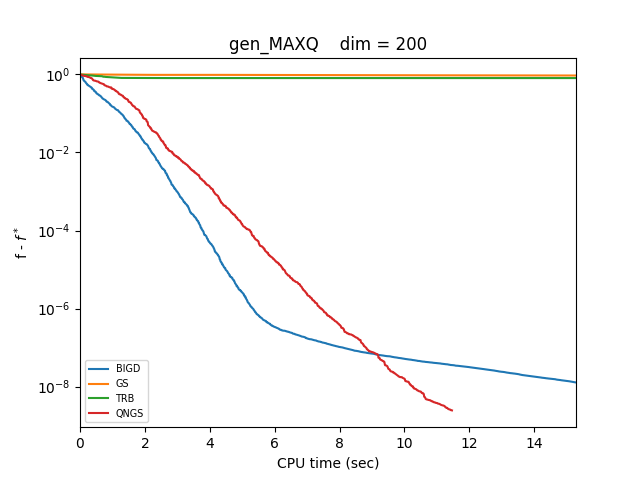}
        \caption{}
    \end{subfigure}
    \hfill
    \begin{subfigure}[b]{0.48\textwidth}
        \includegraphics[width=\textwidth]{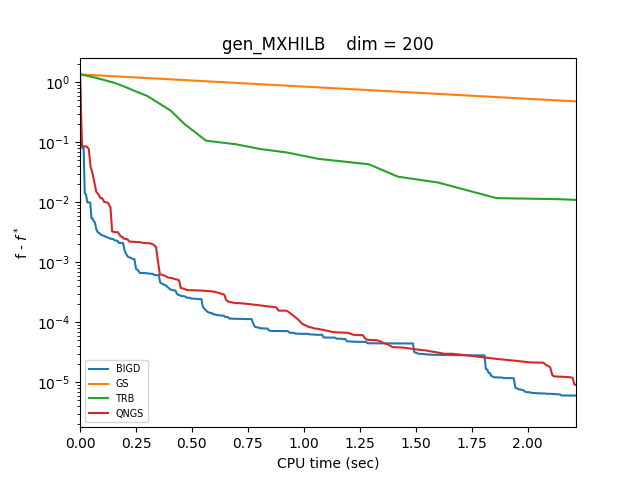}
        \caption{}
    \end{subfigure}
    
    \begin{subfigure}[b]{0.49\textwidth}
        \includegraphics[width=\textwidth]{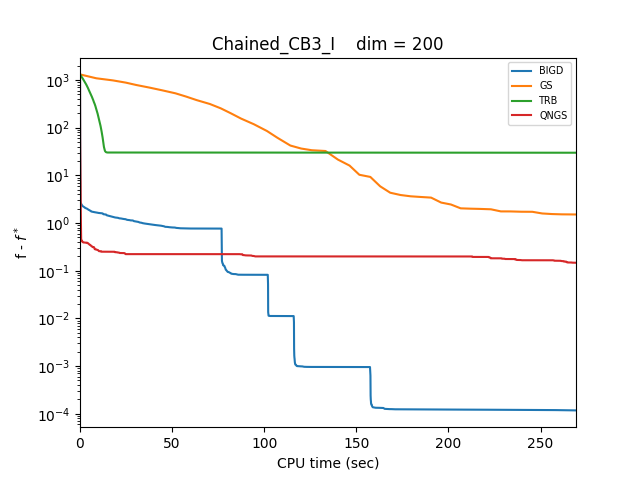}
        \caption{}
    \end{subfigure}
    \hfill
    \begin{subfigure}[b]{0.49\textwidth}
        \includegraphics[width=\textwidth]{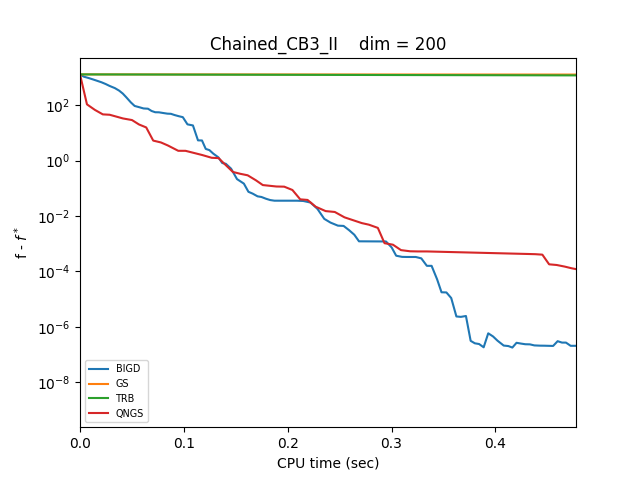}
        \caption{}
    \end{subfigure}
    
    \begin{subfigure}[b]{0.49\textwidth}
        \includegraphics[width=\textwidth]{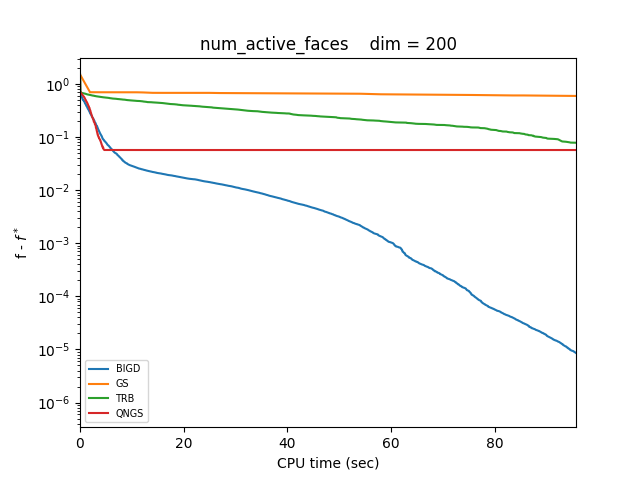}
        \caption{}
    \end{subfigure}
    \hfill
    \begin{subfigure}[b]{0.49\textwidth}
        \includegraphics[width=\textwidth]{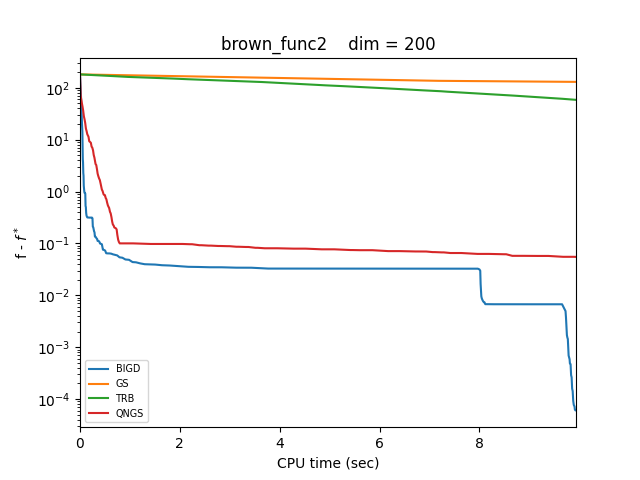}
        \caption{}
    \end{subfigure}
    
    \begin{subfigure}[b]{0.49\textwidth}
        \includegraphics[width=\textwidth]{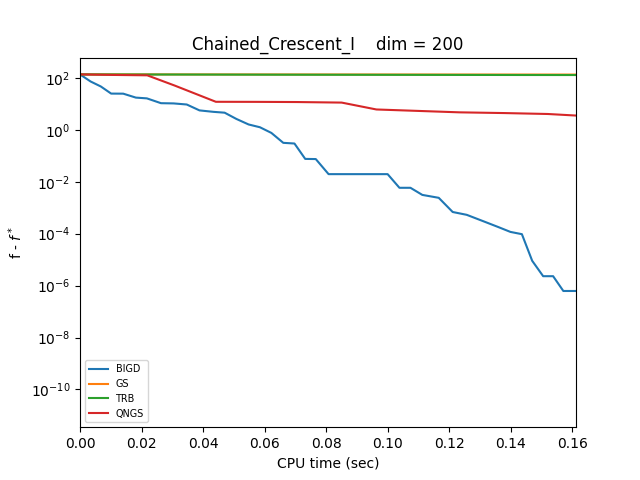}
        \caption{}
    \end{subfigure}
    \hfill
    \begin{subfigure}[b]{0.49\textwidth}
        \includegraphics[width=\textwidth]{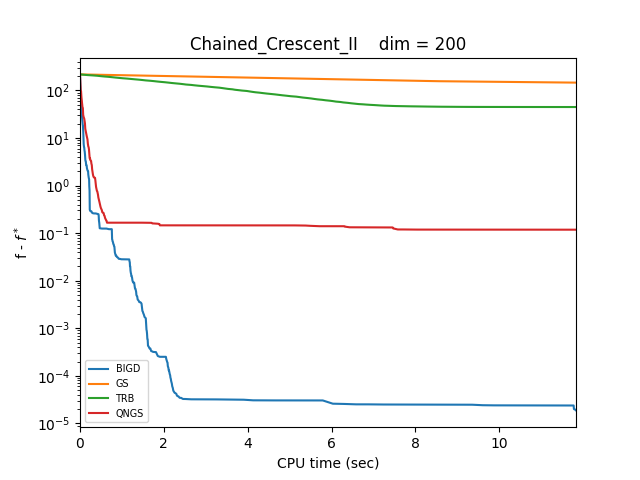}
        \caption{}
    \end{subfigure}
    \caption{Objective value versus computational time for 8 problems with $n=200$ and random initial points}
    \label{fig:obj_dec}
\end{figure}

\begin{figure}[htbp]
    \centering
    \begin{subfigure}[b]{0.48\textwidth}
        \includegraphics[width=\textwidth]{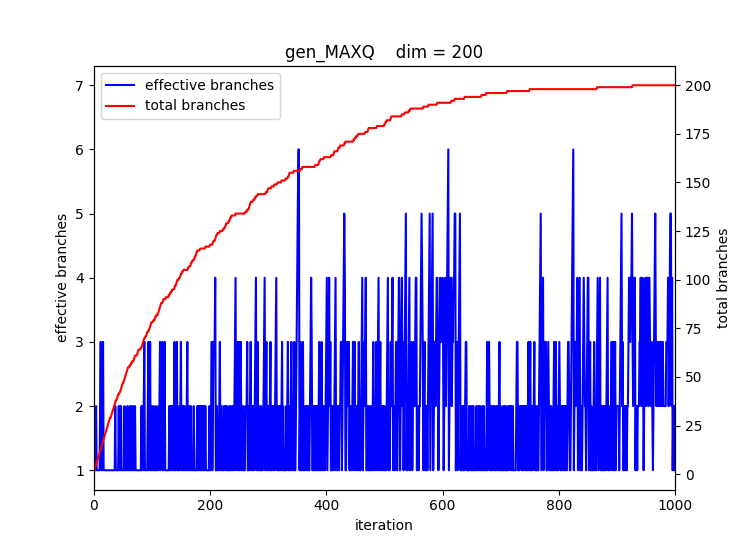}
        \caption{}
    \end{subfigure}
    \hfill
    \begin{subfigure}[b]{0.48\textwidth}
        \includegraphics[width=\textwidth]{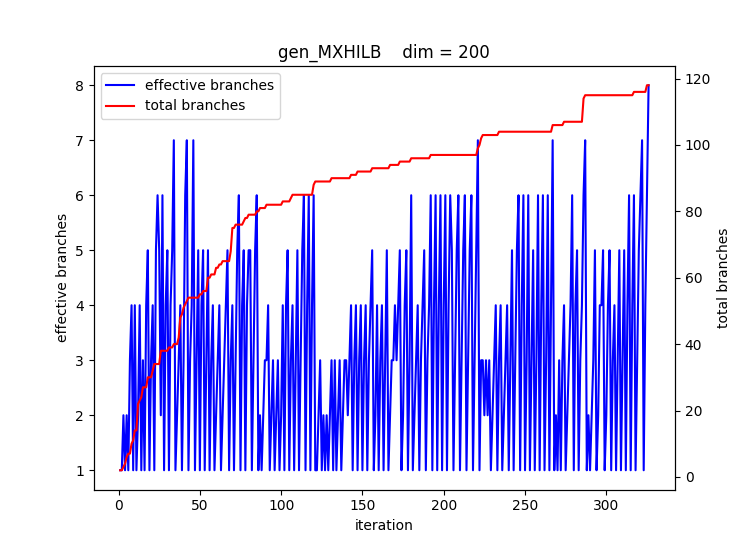}
        \caption{}
    \end{subfigure}
    
    \begin{subfigure}[b]{0.49\textwidth}
        \includegraphics[width=\textwidth]{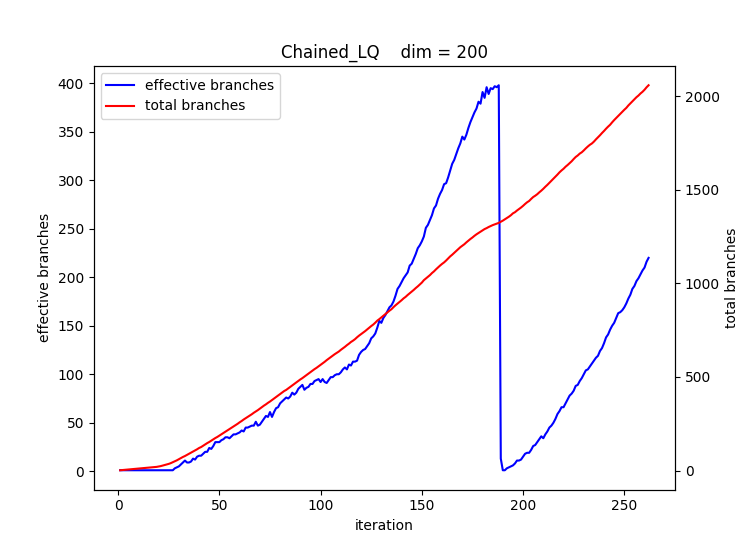}
        \caption{}
    \end{subfigure}
    \hfill
    \begin{subfigure}[b]{0.49\textwidth}
        \includegraphics[width=\textwidth]{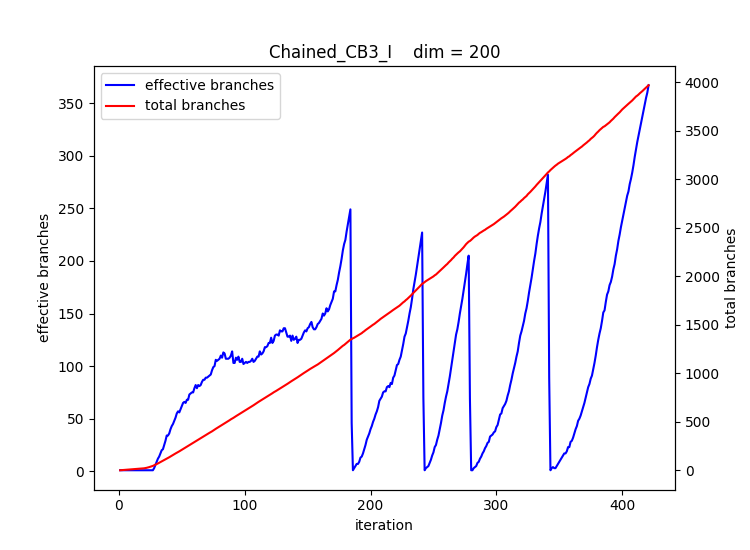}
        \caption{}
    \end{subfigure}
    
    \begin{subfigure}[b]{0.49\textwidth}
        \includegraphics[width=\textwidth]{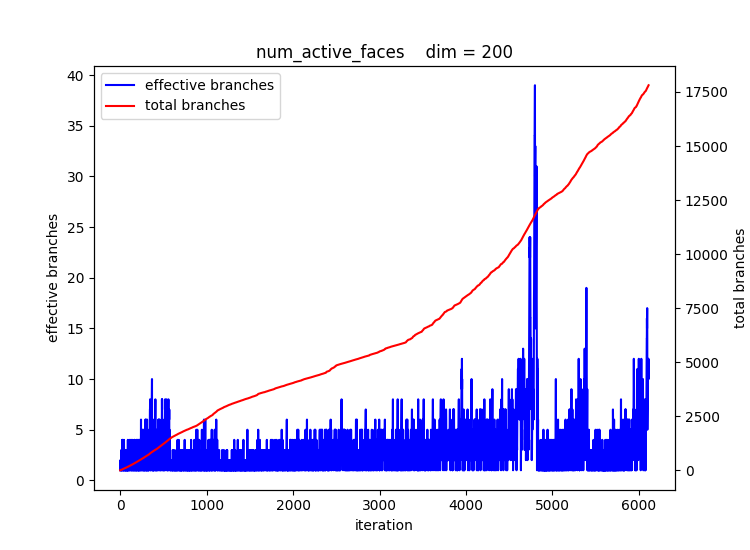}
        \caption{}
    \end{subfigure}
    \hfill
    \begin{subfigure}[b]{0.49\textwidth}
        \includegraphics[width=\textwidth]{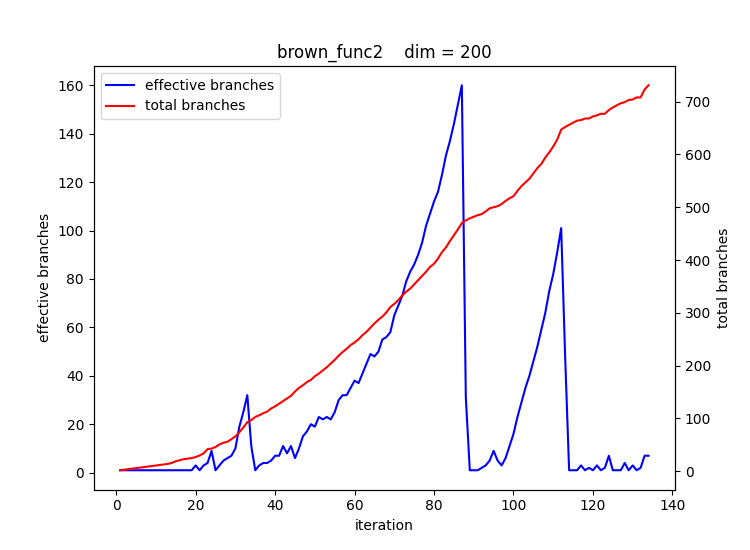}
        \caption{}
    \end{subfigure}
    
    \begin{subfigure}[b]{0.6\textwidth}
        \centering
        \includegraphics[width=\textwidth]{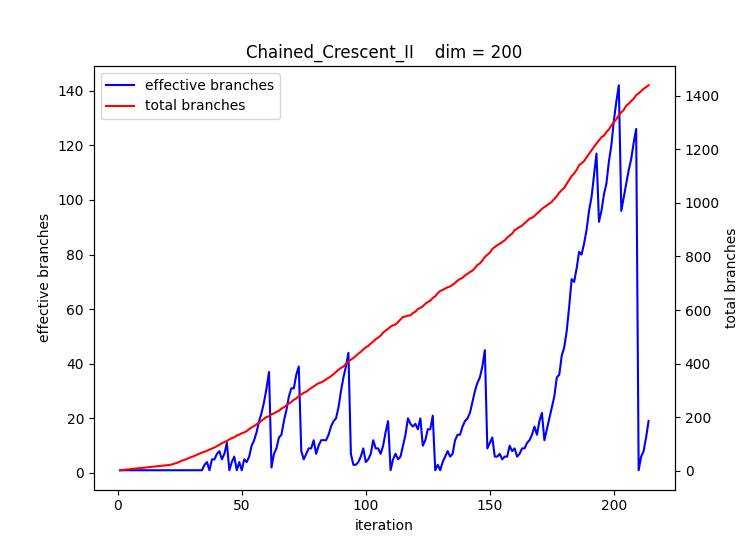}
        \caption{}
    \end{subfigure}
    \caption{Number of effective branches and visited branches for 7 problems with $n=200$ and random initial points}
    \label{fig:act_branch}
\end{figure}

\bibliography{references}
\appendix
\section{Auxiliary results}\label{app:proof} 
\begin{proposition}\label{prop:<u,vi>}
For the set $\{v_i\}^k_{i=1}$ of $k$ vectors in $\mR^m$, consider the quadratic programming problem 
\beq\label{opt:min_lambda_v}
\min_\lambda \norm{\sum^k_{i=1}\lambda_iv_i}^2\;
\mathrm{ s.t. } \sum^k_{i=1}\lambda_i=1,\;\lambda_i\ge 0\;\forall i\in[k].
\eeq
Let $\lambda^*_i$ be an optimal solution of \eqref{opt:min_lambda_v}
and $d^*=\sum^k_{i=1}\lambda^*_iv_i$.
Then the inequality $\inner{v_i}{d^*}\ge\norm{d^*}^2$ holds for each $i\in[k]$.
\end{proposition}
\begin{proof}
Let $C=\conv\{v_1,\ldots,v_k\}$ and then $d^*$ can is the projecting point of $\bs{0}$
on the convex set $C$. Using the projection theorem on a convex set, any point 
$x\in{C}$ can be represented as $x=d^*+u$ with $\inner{u}{d^*}\ge0$.
Therefore, one has $\inner{x}{d^*}\ge\norm{d^*}^2$, and in particular, 
$\inner{v_i}{d^*}\ge\norm{d^*}^2$ for all $i\in[k]$.
\end{proof}
\begin{proposition}\label{prop:d_limit}
Let $\{v_i\}^k_{i=1}$, $\lambda^*$ and $d^*$ be defined the same as in \propref{prop:<u,vi>}.
Suppose for each $i\in[k]$, the is a sequence $\{u^j_i\}^\infty_{j=1}$ of vectors satisfying 
$\lim_ju^j_i=v_i$. Let $\lambda^j$ be an optimal solution of the quadratic program \eqref{opt:min_lambda_v}
after replacing $v_i$ with $u^j_i$ for $i\in[k]$, and let $d^j=\sum^k_{i=1}\lambda^ju^j_i$. Then
the following limit holds: $\lim_j\norm{d^j}=\norm{d^*}$.
\end{proposition}
\begin{proof}
We can derive an upper bound of $\|d^j\|$ as follows:
\bdm
\ba
\|d^j\|&=\norm{\sum^k_{i=1}\lambda^j_iu^j_i}\le\norm{\sum^k_{i=1}\lambda^*_iu^j_i}
\le\norm{\sum^k_{i=1}\lambda^*_iv_i}+\norm{\sum^k_{i=1}\lambda^*_iu^j_i - \sum^k_{i=1}\lambda^*_iv_i} \\
&\le\norm{d^*}+\sum^k_{i=1}\lambda^*_i\norm{u^j_i-v_i}\le\norm{d^*}+\max_i\norm{u^j_i-v_i},
\ea
\edm
where the first inequality follows from the optimality of $d^j$.
On the other hand, $\|d^j\|$ can be lower bounded as:
\bdm
\ba
\|d^j\|&=\norm{\sum^k_{i=1}\lambda^j_iu^j_i}\ge\norm{\sum^k_{i=1}\lambda^j_iv_i}-\sum^k_{i=1}\lambda^j_i\norm{u^j_i-v_i}\\
&\ge\norm{d^*}-\sum^k_{i=1}\lambda^j_i\norm{u^j_i-v_i} \ge\norm{d^*}-\max_i\norm{u^j_i-v_i},
\ea
\edm
where the second inequality is due to the optimality of $d^*$.
Since $\lim_ju^j_i=v_i$, the above two inequalities imply $\lim_j\|d^j\|=\norm{d^*}$.
\end{proof}
\end{document}